\documentclass[11pt]{amsart}
\usepackage[utf8]{inputenc}

\usepackage{amsmath,amssymb,amsfonts,amsthm}
\usepackage{graphicx}
\usepackage[colorlinks=true]{hyperref}
\usepackage{enumerate}
\usepackage{tikz}

\numberwithin{equation}{section}
\newcommand{\R}{\mathbb{R}}
\newcommand{\N}{\mathbb{N}}
\newcommand{\Z}{\mathbb{Z}}
\newcommand{\C}{\mathbb{C}}

\newcommand{\W}{\mathcal{W}}
\newcommand{\CZ}{\mathcal{CZ}}

\newcommand{\beqnn}{\begin{eqnarray*}}
\newcommand{\eeqnn}{\end{eqnarray*}}
\newcommand{\beqn}{\begin{eqnarray}}
\newcommand{\eeqn}{\end{eqnarray}}
\newcommand{\beq}{\begin{equation}}
\newcommand{\eeq}{\end{equation}}

\theoremstyle{plain}
\newtheorem{thm}{Theorem}[section]
\newtheorem{prop}[thm]{Proposition}
\newtheorem{conj}[thm]{Conjecture}
\newtheorem{lem}[thm]{Lemma}
\newtheorem{cor}[thm]{Corollary}
\newtheorem{rmk}[thm]{Remark}
\newtheorem{defi}{Definition}[section]
\newtheorem{exm}{Example}[section]
\begin{document}

\title{Local Centralizer Rigidity near Elements of the Weyl Chamber Flow}
\author{Zhijing Wendy Wang}
\address{Yau Mathematical Sciences Centre, Tsinghua University, Beijing, China}
\email{zj-wang19@mails.tsinghua.edu.cn}
\address{Department of Mathematics, University of Chicago, Chicago, IL 60637, USA}
\email{zhijingw@uchicago.edu}

\maketitle
\begin{abstract}
In this paper, we prove centralizer rigidity near an element of the Weyl chamber flow on a semisimple Lie group. We show that a $C^1$ perturbation of an element of the Weyl chamber flow on a quotient $G/\Gamma$ of an $\R$-split, simple Lie group $G$ either has a centralizer of dimension $0$ or $1$, or is smoothly conjugate to an element of the Weyl chamber flow. We also obtain a general condition for the center-fixing centralizer of a partially hyperbolic diffeomorphism to be a Lie group.
\end{abstract}

\tableofcontents
\section{Introduction}

For a diffeomorphism $f\in \mathrm{Diff}^\infty(M)$ of a Riemannian manifold $M$, we define its $C^r$-centralizer, for $1\le r\le \infty$, to be
\[
\mathcal{Z}^r(f):=\{g\in \mathrm{Diff}^r(M): g\circ f=f\circ g\}.
\]
We say that a diffeomorphism has
\begin{itemize}
 \item \emph{trivial centralizer} if $\mathcal{Z}^\infty(f)=\langle f\rangle;$ 
 \item \emph{virtually trivial centralizer} if $\langle f\rangle$ is a finite-index subgroup of $\mathcal{Z}^\infty(f)$.
\end{itemize}

The study of smooth centralizers started with Smale \cite{smale}, who conjectured that centralizers of ``generic'' diffeomorphisms are trivial. The conjecture was partially answered in the affirmative in \cite{BCW}, where the authors showed that a residual set of $C^1$ diffeomorphisms of a compact smooth manifold has trivial centralizer.

Here, we are interested in the problem of describing the centralizer completely around an algebraic model and showing that the only reason for the centralizer to be non-trivial is that the diffeomorphism is very special, or, in other words, that ``big centralizer implies being algebraic.'' This is called the (local) \emph{centralizer rigidity} problem.

The centralizer rigidity problem can also be seen as a generalization of the rigidity of higher-rank actions under perturbation, which was studied by Katok--Spatzier \cite{KSpat}, Damjanović \cite{DK}, Vinhage \cite{Vinhage}, Wang \cite{VinWang}, and others. Instead of perturbing an entire action, the centralizer rigidity question asks whether the condition that a single element is perturbed is enough to establish rigidity of the whole action.

In this paper, we consider perturbations of affine diffeomorphisms on the quotient $G/\Gamma$ of a semisimple Lie group, for example, a diagonal matrix acting by left multiplication on $\mathrm{SL}_n\R/\Gamma$. Here we take $G$ to be a connected, $\R$-split semisimple Lie group with no compact factor. We also require $G$ to be \emph{genuinely higher rank}, meaning that each of its simple factors has rank at least $2$. We require $\Gamma$ to be a cocompact lattice in $G$.

We pick the diffeomorphisms to be elements of the Weyl chamber flow, which is an $\R^k$-action on $G/\Gamma$ given by left multiplication by elements of its Cartan subgroup. An element $f_0$ of the Weyl chamber flow is called \emph{generic} if the corresponding element in the Cartan subalgebra does not lie on the Weyl chamber walls. See Section~\ref{semisimple} for more detailed definitions of this terminology.

In the simplest example where $G=\mathrm{SL}_3\R$, the Weyl chamber flow is just the $\R^2$-action given by left multiplication by diagonal matrices
\[
\begin{pmatrix}
 e^{t_1} &&\\ &e^{t_2}&\\ &&e^{t_3}
\end{pmatrix},
\qquad t_1,t_2,t_3\in \R,\; t_1+t_2+t_3=0,
\]
on $\mathrm{SL}_3\R/\Gamma$. In this case, an element
\[
f_0=\begin{pmatrix}
 e^{t_1} &&\\ &e^{t_2}&\\ &&e^{t_3}
\end{pmatrix}
\]
is generic if and only if $t_i\neq t_j$ for any $1\le i<j\le 3$.

The centralizer $\mathcal{Z}^\infty(f_0)$ of a generic element $f_0$ in the Weyl chamber flow contains $\R^k$, since it commutes with every other element of the Weyl chamber flow. We show that if $G$ is simple, then for a $C^1$ perturbation of $f_0$, the diffeomorphism either has a centralizer of dimension at most $1$, or is smoothly conjugate to an element of the Weyl chamber flow. We note that for an $\R$-split Lie group, an open dense subset of the elements of the Weyl chamber flow is generic.

\begin{thm}\label{cenrig}
Let $X=G/\Gamma$, where $G$ is a connected, $\R$-split simple Lie group of rank at least $2$ with finite center, and $\Gamma$ is a cocompact lattice in $G$.

Let $f_0: X\to X$ be a generic element of the Weyl chamber flow on $G/\Gamma$. Then for any $C^\infty$ diffeomorphism $f:X\to X$ that is in a bounded $C^2$-neighborhood and a sufficiently small $C^1$-neighborhood of $f_0$, the smooth centralizer $\mathcal{Z}^\infty(f)$ is
\begin{itemize}
 \item either a (not necessarily connected) Lie group of dimension $0$ or $1$,
 \item or virtually $\R^k$, where $k$ is the rank of $G$.
\end{itemize}
In the latter case, $f$ is $C^\infty$-conjugate to an element of the Weyl chamber flow.

\end{thm}

In the case that $G$ is semisimple, the condition that the centralizer has dimension at least $2$ is not enough to ensure that $f$ is essentially affine. For example, take
\[
X=(G_1/\Gamma_1)\times (G_2/\Gamma_2),
\]
and let $f_0=(a_1,a_2)$, where $a_i$ is a generic element of the Weyl chamber flow on $G_i/\Gamma_i$. Then we may take a twisted perturbation
\[
f: X\to X,\qquad (x_1,x_2)\mapsto (a(x_2)\cdot x_1,\, a_2 x_2),
\]
where $a:G_2/\Gamma_2\to D(G_1)$ is a $C^1$ perturbation of the constant map $a_1$. Under this construction, the centralizer $\mathcal{Z}^\infty(f)$ has dimension $\mathrm{rank}(G_1)+1>2$. Instead, we prove the following centralizer rigidity result for semisimple Lie groups.

\begin{thm}\label{cenrigsemi}
Let $X=G/\Gamma$, where $G$ is a connected, $\R$-split semisimple Lie group with finite center, and $\Gamma$ is a cocompact lattice in $G$. Suppose the Lie algebra $\mathfrak{g}=\oplus \mathfrak g_i$, where each simple component $\mathfrak g_i$ has real rank $k_i\ge 2$
and the rank of $G$ is $k=\sum_i k_i$.

Let $f_0: X\to X$ be a generic element of the Weyl chamber flow on $G/\Gamma$. Then for any $C^\infty$ diffeomorphism $f:X\to X$  that is in a bounded $C^2$-neighborhood and a sufficiently small $C^1$-neighborhood of $f_0$, the smooth centralizer $\mathcal{Z}^\infty(f)$ is
\begin{itemize}
 \item either a (not necessarily connected) Lie group of dimension at most $k-\min_ik_i+1$,
 \item or virtually $\R^k$, where $k$ is the rank of $G$.
\end{itemize}
In the latter case, $f$ is $C^\infty$-conjugate to an element of the Weyl chamber flow.
\end{thm}

An ingredient in the proofs of Theorems~\ref{cenrig} and~\ref{cenrigsemi} is a result of independent interest, which we state as Theorem~\ref{czlieg} in Section~\ref{cenfix}. It shows that for a ``nice enough'' partially hyperbolic diffeomorphism, the center-fixing centralizer is always a Lie group.

One application of Theorem~\ref{czlieg} is that we remove the volume-preserving assumptions in the centralizer rigidity results for discretized Anosov flows in \cite{DWX} and \cite{BG}; see Theorems~\ref{1dcen1} and~\ref{1dcen2}.

\subsection{Some historical remarks and further questions}

The notion of centralizer rigidity for partially hyperbolic diffeomorphisms was first introduced by Damjanović, Wilkinson, and Xu \cite{DWX}, where the authors studied perturbations of discretized geodesic flows and affine toral automorphisms. In particular, they showed that a volume-preserving perturbation of a discretized geodesic flow of a negatively curved locally symmetric manifold either has virtually trivial centralizer or embeds into a smooth flow.
Gan et al.\ \cite{GSXZ} proved that a partially hyperbolic diffeomorphism on the torus $\mathbb{T}^3$ homotopic to an Anosov linear map either has virtually trivial centralizer or is smoothly conjugate to a linear map.
Damjanović, Wilkinson, and Xu \cite{DWX23} later established centralizer rigidity for some automorphisms on nilmanifolds.

Most of the work in this direction, except for the results on discretized geodesic flows, has been carried out in the setting of nilmanifolds, especially tori. In this paper, we extend the landscape to $\R$-split simple Lie groups. This is also the first result establishing centralizer rigidity for partially hyperbolic systems with center leaves of arbitrary dimension; all previous results were for systems with center dimension $1$ or $2$.

In a recent work by Damjanović, Wilkinson, Wu, and Xu \cite{DWWX}, the authors made the following conjectures on centralizer rigidity of an affine diffeomorphism $f_0=L_a\circ \Psi$ acting on a homogeneous space $X=G/\Gamma$, where $G$ is a Lie group, $\Gamma<G$ is a cocompact discrete subgroup, $L_a$ is left multiplication by some $a\in G$, and $\Psi$ is a $G$-automorphism preserving $\Gamma$.

\begin{conj}[Conjecture 1, \cite{DWWX}]\label{DWXconj1}
Let $f_0:X\to X$ be an affine diffeomorphism that satisfies the $K$-property. Then for every $f\in \mathrm{Diff}^\infty(X)$ sufficiently $C^1$-close to $f_0$, the centralizer $\mathcal{Z}^\infty(f)$ is a $C^0$-closed Lie subgroup of $\mathrm{Diff}^\infty(X)$.
\end{conj}

\begin{conj}[Conjecture 2, \cite{DWWX}]\label{DWXconj3}
Let $f_0:X\to X$ be an affine diffeomorphism that satisfies the $K$-property, and suppose $\mathcal{Z}^\infty(f_0)$ has no rank-$1$ factor. Then for every $f\in \mathrm{Diff}^\infty(X)$ sufficiently $C^1$-close to $f_0$, if the centralizer $\mathcal{Z}^\infty(f)$ has no rank-$1$ factor, then $f$ is smoothly conjugate to an affine diffeomorphism.
\end{conj}

\begin{conj}[Conjecture 3, \cite{DWWX}]\label{DWXconj2}
Let $f_0:X\to X$ be an affine diffeomorphism that satisfies the $K$-property, and suppose $\mathcal{Z}^\infty(f_0)$ has no rank-$1$ factor. Then for every $f\in \mathrm{Diff}^\infty(X)$ sufficiently $C^1$-close to $f_0$, if the centralizer $\mathcal{Z}^\infty(f)\simeq \mathcal{Z}^\infty(f_0)$, then $f$ is smoothly conjugate to an affine diffeomorphism.
\end{conj}

Here, a group action is said to have a rank-$1$ factor if it admits a group action of $\Z$ or $\R$, or their compact extensions, as a factor. The simplest example of a group action with a rank-$1$ factor is when the group is virtually $\Z$ or $\R$. If Smale's conjecture is true, then a generic perturbation of a diffeomorphism would have trivial centralizer and hence a centralizer with a rank-$1$ factor.

Theorems~\ref{cenrig} and~\ref{cenrigsemi} give an affirmative answer to Conjectures~\ref{DWXconj1} and~\ref{DWXconj2} for perturbations of generic elements of the Weyl chamber flow on quotients of $\R$-split semisimple Lie groups $G/\Gamma$. They give strong evidence toward Conjecture~\ref{DWXconj3} for simple $\R$-split Lie groups.

It remains an open question whether Conjecture~\ref{DWXconj3} holds for quotients of $\R$-split semisimple Lie groups. For example, in Theorem~\ref{cenrigsemi}, if $\Gamma$ is an irreducible lattice and the centralizer of $f$ has dimension at least $2$, is $f$ necessarily conjugate to an element of the Weyl chamber flow?

Some of the methods in this paper can also be extended to twisted Weyl chamber flows, non-split semisimple Lie groups, and non-generic elements of the Weyl chamber flow, assuming larger centralizers in the latter two cases. These results will appear in separate papers.

\subsection{Organization and outline of the proof}\label{outl}

In Section~\ref{prelim}, we provide some preliminaries on basic Lie theory and partially hyperbolic dynamics.

Let $f_0$ be a generic element of the Weyl chamber flow, and let $f$ be a $C^1$-small perturbation as in Theorem~\ref{cenrig}. The action of $f_0$ fixes each coset of the Cartan subgroup, which we call the center leaves of $f_0$. Normal to the directions of the Cartan subgroup, we show that $f_0$ acts by contraction or expansion. These properties characterize $f_0$ as being normally hyperbolic. Furthermore, the action of $f_0$ restricted to the Cartan subgroup is simply a translation map. See Section~\ref{semisimple} for more details.

A classical result in partially hyperbolic dynamics (see Theorem~\ref{leafconj}) shows that a $C^1$ perturbation $f$ of $f_0$ also fixes each leaf of a foliation $\W^c_f$ and acts by contraction and expansion normal to that foliation. Let $g\in \mathcal{Z}^r(f)$ be any diffeomorphism commuting with $f$. It is not hard to show that $g$ must also preserve the foliation $\W^c_f$. Therefore, we may map $g\in \mathcal{Z}^r(f)$ to its action on the leaf space of $\W^c_f$, yielding a short exact sequence
\[
1\to \CZ^r(f)\to \mathcal{Z}^r(f)\to \mathcal{Z}^r(f)/\CZ^r(f)\to 1.
\]

Here, we call the kernel
\[
\CZ^r(f):=\{g\in \mathrm{Diff}^r(M): f\circ g=g\circ f,\; g(\W_f^c(x))=\W_f^c(x)\ \forall\, x\in M\}
\]
the center-fixing $C^r$ centralizer of $f$. The quotient space $\mathcal{Z}^r(f)/\CZ^r(f)$ is the space of the action of $\mathcal{Z}^r(f)$ on the leaf space $X/\W^c_f$.

In Section~\ref{cenfix}, we prove that the center-fixing centralizer $\CZ^r(f)$ is a Lie group under a more general condition. Applying this result to discretized Anosov flows, we remove the volume-preserving assumptions in \cite{DWX} and \cite{BG}.

We then begin the proofs of Theorems~\ref{cenrig} and~\ref{cenrigsemi}.

In Section~\ref{centgp}, we prove that in our setting the space of actions on the leaf space, $\mathcal{Z}^r(f)/\CZ^r(f)$, must be finite. Therefore, $\mathcal{Z}^r(f)$ is virtually a Lie group.

In Section~\ref{dichcen}, we prove that elements of $\CZ^r(f)$ also act ``like translations'' on the center leaves, similar to the way $f_0$ acts on the cosets of the Cartan subgroup. Using this fact and the fact that $\CZ^r(f)$ is a Lie group, we conclude that the centralizer either contains two diffeomorphisms that act like translations in different directions, or has dimension at most $1$ when $G$ is simple.

In Section~\ref{eph}, we prove that in the case where the centralizer contains two diffeomorphisms that act ``like translations'' in different directions, $\CZ^r(f)$ must contain a higher-rank action that is ``topologically close'' to a partially hyperbolic one. In particular, we show that suitable elements of the centralizer have the same type of contraction and expansion as elements of the Weyl chamber flow along suitable stable and unstable foliations; and they have extremely slow contraction and expansion rate along the center foliation.

In Section~\ref{wcf}, we prove that this higher-rank action must be conjugate to an algebraic one, following the methods of \cite{DK} and \cite{VinWang}.

\subsection{Acknowledgments}

The author would like to thank her advisor, Amie Wilkinson, for many useful discussions, guidance, and thorough revision of this work. The author also thanks Disheng Xu, Danijela Damjanović, Sven Sandfeldt and Kurt Vinhage for illuminating conversations, and thanks Aaron Brown and her advisor Jinxin Xue for many useful comments. The author thanks the anonymous referee for many helpful comments and suggestions that greatly improved the quality of this paper.

\section{Preliminaries}\label{prelim}

\subsection{Semisimple Lie groups and Weyl chamber flows}\label{semisimple}
Let $G$ be a $\R$-split semisimple Lie group with Lie algebra $\mathfrak{g}$. 

A \emph{Cartan subalgebra} $\mathfrak{h}$ of $\mathfrak{g}$ is a maximal commutative subalgebra of $\mathfrak{g}$. We also denote by $D=\exp(\mathfrak{h})$ the corresponding (identity component of the) maximal abelian subgroup, called the Cartan subgroup of $G$. For $G=\mathrm{SL}_n\R$, $D$ is conjugate to the subgroup that consists of the diagonal matrices.

Restricting the adjoint action to $\mathfrak{h}$, $\mathrm{ad}|_{\mathfrak{h}}$ decomposes the Lie algebra $\mathfrak{g}$ into eigenspaces
\[
\mathfrak{g}=\bigoplus_{\lambda\in \mathfrak{h}^*}\mathfrak{g}_\lambda,
\]
where each eigenspace is given by
\[
\mathfrak{g}_\lambda=\{x\in \mathfrak{g}:\ \mathrm{ad}(h)x={\lambda(h)}x,\ \forall\, h\in \mathfrak{h}\}.
\]

We denote by
\[
\Phi=\{\lambda\in \mathfrak{h}^*:\ \mathfrak{g}_\lambda\neq 0,\ \lambda\neq 0\}\subset \mathfrak{h}^*
\]
the set of non-zero eigenvalues of the adjoint representation $\mathfrak{h}\to \mathfrak{gl}(\mathfrak{g})$, and we call $\Phi$ the \emph{root system} of $G$.

The eigenvalues, which are linear functions on $\mathfrak{h}$, separate $\mathfrak{h}\simeq \R^k$ into different Weyl chambers. The kernels of the eigenvalues $\ker(\lambda)\subset \mathfrak{h}$, $\lambda\in \Phi$, are called the \emph{Weyl chamber walls} in $\mathfrak{h}$, and the connected components of
\[
\mathfrak{h}\setminus \bigcup_{\lambda\in \Phi}\ker(\lambda)
\]
are called the \emph{Weyl chambers}. An element $h\in \mathfrak{h}$ is called \emph{generic} if it does not lie on any of the Weyl chamber walls.

Let $\Gamma$ be a discrete subgroup of $G$. The Weyl chamber flow on $G/\Gamma$ is the action
\[
\mathfrak{h}\to \mathrm{Diff}^\infty(G/\Gamma),\qquad v\mapsto L_{\exp(v)},
\]
where $L_{\exp(v)}$ denotes left multiplication by $\exp(v)$ on $G/\Gamma$.

A simple and concrete example is given by $G=\mathrm{SL}_3\R$.

\begin{exm}
Let $G=\mathrm{SL}_3\R$. Then the Lie algebra is
\[
\mathfrak{sl}_3\R=\{A\in M_{3\times 3}(\R): \mathrm{tr}\,A=0\}.
\]

One choice of a maximal diagonalizable subalgebra is
\[
\mathfrak{h}=\left\{
\begin{pmatrix}
 t_1 & &\\
 & t_2&\\
 && t_3
\end{pmatrix}:\ t_1+t_2+t_3=0,\ t_1,t_2,t_3\in \R
\right\}.
\]

The eigenspaces of the adjoint action are
\[
\mathfrak{g}_{ij}=\R e_{ij},\qquad 1\le i\neq j\le 3,
\]
where $e_{ij}$ is the matrix with $1$ in the $(i,j)$-entry and zeros elsewhere. The corresponding eigenvalues $\lambda_{ij}$ are given by
\[
\lambda_{ij}\!\left(
\begin{pmatrix}
 t_1 & &\\
 & t_2&\\
 && t_3
\end{pmatrix}
\right)=t_i-t_j.
\]

The Weyl chamber walls are the three lines $\{t_i=t_j\}$ in the Cartan subalgebra
\[
\mathfrak{h}\simeq \{(t_1,t_2,t_3)\in \R^3:\ t_1+t_2+t_3=0\}\simeq \R^2,
\]
which partition $\mathfrak{h}$ into six Weyl chambers.
\end{exm}

Now consider the action of left multiplication by $f_0=L_{\exp(v)}$ on $G$, where $v$ is a generic element of $\mathfrak{h}$ (that is, $v$ does not lie on any Weyl chamber wall). Let $G_\lambda=\exp(\mathfrak{g}_\lambda)$ be the subgroup generated by the eigenspace $\mathfrak{g}_\lambda$. Then $f_0$ preserves the foliation given by cosets of $G_\lambda$. We shall see that $f_0$ also expands or contracts the cosets of $G_\lambda$ exponentially.

\begin{prop}
For any $x,y$ in the same coset of $G_\lambda$, the distance between $x$ and $y$ is contracted or expanded exponentially under iteration of $f_0$, depending on the sign of $\lambda(v)$.
\end{prop}

\begin{proof}
Let $x,y\in G$ lie in the same coset of $G_\lambda$. Then there exists $u\in \mathfrak{g}_\lambda$ such that $x=\exp(u)y$. We compute
$$
f_0^n(x)f_0^n(y)^{-1}
=\exp(nv)\, x y^{-1}\, \exp(-nv)
=\exp(\mathrm{Ad}(\exp(nv))u)
$$$$=\exp(\exp(\mathrm{ad}(nv))u)
=\exp(\exp(n\lambda(v))u).
$$
Therefore, under iteration of $f_0$, the distance between $x$ and $y$ in the same orbit of $G_\lambda$ is contracted or expanded exponentially, depending on the sign of $\lambda(v)$.
\end{proof}

We also note that since $\exp(v)$ commutes with the Cartan subgroup $D$, the map $f_0$ acts on each coset of $D$ by a translation, neither expanding nor contracting the cosets of $D$.

As we shall see later, these behaviors exactly characterize those of a normally hyperbolic diffeomorphism.

\subsection{Partial hyperbolicity and regularity}\label{prel:ph}
We recall the definitions of dominated splitting, partial hyperbolicity, and some results from \cite{HPS} on leaf conjugacy.

\begin{defi}[Dominated Splitting]
Let $M$ be a Riemannian manifold, and let $f:M\to M$ be a diffeomorphism. A \emph{dominated splitting} of $f$ is a decomposition
\[
TM=E^1\oplus E^2\oplus \cdots \oplus E^k
\]
that satisfies the following conditions:
\begin{itemize}

\item the splitting $TM=E^1\oplus E^2\oplus \cdots \oplus E^k$ is $Df$-invariant, i.e.\ $Df(E^i(x))=E^i(f(x))$ for any $x\in M$;

\item there exist constants $0<\lambda_i<1$ and $C>0$ such that for any $n\in \N$ and $x\in M$,
\[
\|Df^n(x)v\|\le C\lambda_i^n \|Df^n(x)u\|,
\]
for any $v\in E^i(x)$ and $u\in E^{i+1}(x)$ with $\|u\|=\|v\|=1$, and any $1\le i\le k-1$.
\end{itemize}
\end{defi}

\begin{defi}[Partially Hyperbolic Diffeomorphism]
A diffeomorphism $f:M\to M$ of a Riemannian manifold $M$ is \emph{partially hyperbolic} if $f$ has a dominated splitting
\[
TM=E^s\oplus E^c\oplus E^u,
\]
and there exist constants $0<\mu<1$ and $C>0$ such that for any $n\in \N$ and $x\in M$,
\[
\|Df^n(x)v\|\le C\mu^n\|v\| \quad \text{for any } v\in E^s(x),
\]
and
\[
\|Df^{-n}(x)u\|\le C\mu^n\|u\| \quad \text{for any } u\in E^u(x).
\]

In this paper, we also assume that the bundles $E^s$ and $E^u$ are both nontrivial.
\end{defi}

In general, for a partially hyperbolic diffeomorphism, the stable and unstable bundles $E^s$ and $E^u$ are uniquely integrable to $f$-invariant stable and unstable foliations $\W^s$ and $\W^u$. However, $E^c$ is generally not tangent to a foliation. A partially hyperbolic diffeomorphism is normally hyperbolic if there exists an invariant center foliation tangent to $E^c$.

\begin{defi}[Normally hyperbolic]
A partially hyperbolic diffeomorphism $f:M\to M$ is said to be \emph{normally hyperbolic} with respect to a foliation $\mathcal{F}$ if $f$ preserves $\mathcal{F}$ and $T\mathcal{F}=E^c$.

For $r\ge 1$, we say that $f$ is \emph{$r$-normally hyperbolic} if it is normally hyperbolic and there exists $k\in \N$ such that
\[
\sup_{x\in M}\|D_xf^k|_{E^s}\|\cdot \|(D_xf^k|_{E^c})^{-1}\|^r<1
\]
and
\[
\sup_{x\in M}\|(D_xf^k|_{E^u})^{-1}\|\cdot \|D_xf^{k}|_{E^c}\|^r<1.
\]
\end{defi}

We note that $1$-normal hyperbolicity is equivalent to normal hyperbolicity, and that $r$-normal hyperbolicity is a $C^1$-open condition.

For a normally hyperbolic diffeomorphism $(f,\mathcal{F})$, heuristically, we may view $f$ as being Anosov on the leaf space $M/\mathcal{F}$. Therefore, we may expect structural stability of $f$ up to the leaves of $\mathcal{F}$. Under suitable conditions, we may expect perturbations of $f$ to be $C^0$ conjugate to $f$ modulo the leaves of $\mathcal{F}$, and we say that such a perturbation is leaf conjugate to $f$. We now define the notion of leaf conjugacy.

\begin{defi}[Leaf conjugacy]
Suppose $(f,\mathcal{F}_f)$ and $(g,\mathcal{F}_g)$ are two diffeomorphisms of $M$ with invariant foliations $\mathcal{F}_f$ and $\mathcal{F}_g$, respectively.

Then $(f,\mathcal{F}_f)$ and $(g,\mathcal{F}_g)$ are said to be \emph{leaf conjugate} via a leaf conjugacy $h\in \mathrm{Homeo}(M)$ if
\[
h(\mathcal{F}_f(x))=\mathcal{F}_g(h(x)) \quad \text{and} \quad
h(f(\mathcal{F}_f(x)))=g(\mathcal{F}_g(h(x)))
\]
for any $x\in M$.
\end{defi}

A notion slightly stronger than normal hyperbolicity is dynamical coherence.

\begin{defi}[Dynamically Coherent]
A partially hyperbolic diffeomorphism $f$ is said to be \emph{dynamically coherent} if there exist two $f$-invariant foliations $\W^{cs}$ and $\W^{cu}$ that are tangent to $E^{cs}=E^c\oplus E^s$ and $E^{cu}=E^c\oplus E^u$, respectively.

If $f$ is dynamically coherent, we define the center foliation of $f$ to be $\W^c=\W^{cs}\cap \W^{cu}$.
\end{defi}

By definition, $\W^c$ is an $f$-invariant foliation tangent to $E^c$ at every point, and thus $(f,\W^c)$ is normally hyperbolic. We note that here we do not require $E^{cs}$ or $E^{cu}$ to be uniquely integrable, and $E^c$ also does not have to be uniquely integrable to $\W^c$. (Although for the applications of this paper, the bundles are uniquely integrable). See \cite{BWcoh} for more discussion of this subject.

We now introduce a central result concerning leaf-wise structural stability of normally hyperbolic diffeomorphisms. A more general statement and proof can be found in \cite{BWcoh} and \cite{PSWholrev}.

\begin{thm}[Leaf-wise structural stability, see Theorem 7.1 of \cite{HPS}]\label{leafconj}
Let $f$ be a $C^r$ diffeomorphism on $M$ preserving a foliation $\mathcal{F}$. If $(f,\mathcal{F})$ is $r$-normally hyperbolic, then the leaves of $\mathcal{F}$ are uniformly $C^r$, and the bundles $E^u$ and $E^s$ are uniquely integrable with leaves as smooth as $f$.

Moreover, if the splitting $TM=E^u\oplus E^c\oplus E^s$ is $C^1$, then $f$ is $C^1$-stably dynamically coherent, and there exists a bi-H\"older leaf conjugacy $\phi$ from $f$ to any $C^1$-perturbation $f_1$ of $f$ that is $C^1$-close to identity, and uniformly $C^r$ when restricted to the center leaves $\W^c_f$.
\end{thm}

Applying this result to the diffeomorphisms of $G/\Gamma$ that we consider, we see that any sufficiently small perturbation $f$ of an element of the Weyl chamber flow $f_0$ has the properties described above.

\begin{prop}
Let $1\le r\le \infty$. Let $f_0$ be a generic element of the Weyl chamber flow, and let $f\in \mathrm{Diff}^r(X)$ be a $C^1$-perturbation of $f_0$. Then $E^u_f$ and $E^s_f$ are uniquely integrable, $f$ is dynamically coherent and $r$-normally hyperbolic, and there exists a bi-H\"older leaf conjugacy from $f_0$ to $f$ that is uniformly $C^r$ on the center leaves. Moreover, any commuting diffeomorphism $g\in \mathcal Z^r(f)$ preserves $\W^{*}_f$ for $*=s,u,cs,cu,c$.
\end{prop}

\begin{proof}
By the analysis in Section~\ref{semisimple}, $f_0$ is $r$-normally hyperbolic with respect to the foliation $\W^c_{f_0}$ given by cosets of $D$. Therefore, applying Theorem~\ref{leafconj} yields dynamical coherence and the H\"older leaf conjugacy. Moreover, since $\W^c_f$ is plaque expansive, by the dynamical characterization of $\W^{*}_f$ for $*=cs,cu,s,u$, any diffeomorphism $g$ that commutes with $f$ preserves these foliations, see Theorem 7.1 and surrounding discussions in \cite{HPS}. 
\end{proof}


\subsection{Accessibility}
We now recall some definitions and facts about $su$-holonomies of partially hyperbolic diffeomorphisms.

In this paper, we shall mainly consider holonomies along the center leaves.

An \emph{$su$-path} of $f$ is a piecewise $C^1$ curve $\gamma:[0,1]\to M$ with each piece contained in either $\W^s$ or $\W^u$, i.e.\ there exists $k\in \N$ and
\[
0=t_0<t_1<t_2<\cdots<t_k=1
\]
such that $\gamma([t_i,t_{i+1}])\subset \W^s(\gamma(t_i))$ or $\gamma([t_i,t_{i+1}])\subset \W^u(\gamma(t_i))$. In this case, the $su$-path $\gamma$ is said to be $k$-legged. Sometimes we also denote an $su$-path by its endpoints
\[
[\gamma(0),\gamma(t_1),\ldots,\gamma(t_k)].
\]

\begin{defi}[Accessibility]
A partially hyperbolic diffeomorphism $f$ is said to be \emph{accessible} if for any $x,y\in M$, there exists an $su$-path from $x$ to $y$.
\end{defi}

For partially hyperbolic systems with center dimension $1$ or $2$, accessibility is known to be open by \cite{Didier} and \cite{AV}. Stable accessibility (under $C^1$-small perturbations) is also shown to be $C^r$-dense, for $r\ge 1$, in \cite{RHU} for center dimension $1$, and $C^1$-dense for arbitrary center dimension in \cite{DW}. Furthermore, if the center foliation of a partially hyperbolic diffeomorphism $f$ is smooth, then accessibility is $C^1$-open around $f$ by Proposition~1.4 of \cite{GPS}.

In our case, since $f_0$ is highly accessible due to the Lie algebra structure and the center leaves are smooth, we can apply Proposition~1.4 of \cite{GPS} to show that accessibility is open around $f_0$.

\begin{prop}\label{faccess}
Let $f_0$ be a generic element of the Weyl chamber flow on $X=G/\Gamma$. Then any $C^1$ perturbation $f$ of $f_0$ is accessible. Furthermore, for any $x,y\in X$ and any $f_0$-$su$-path $\gamma_0$ connecting $x$ and $y$, there exists an $f$-$su$-path $\gamma$ connecting $x$ and $y$ that is close to $\gamma_0$.
\end{prop}

\begin{defi}[Center bunching]
We say that a partially hyperbolic diffeomorphism $f:M\to M$ is \emph{center $r$-bunched} if there exists $k\ge 1$ such that
\[
\sup \|D_xf^k|_{E^s}\|\cdot \|(D_xf^k|_{E^c})^{-1}\|^r<1,
\]
\[
\sup \|(D_xf^k|_{E^u})^{-1}\|\cdot \|D_xf^k|_{E^c}\|^r<1,
\]
\[
\sup \|D_xf^k|_{E^s}\|\cdot \|(D_xf^k|_{E^c})^{-1}\|\cdot \|D_xf^k|_{E^c}\|^r<1,
\]
and
\[
\sup \|(D_xf^k|_{E^u})^{-1}\|\cdot \|(D_xf^k|_{E^c})^{-1}\|^r\cdot \|D_xf^k|_{E^c}\|<1.
\]

We say that $f$ is \emph{center bunched} if it is center $1$-bunched.
\end{defi}

\begin{rmk}
We note that the first two inequalities above are the same as $r$-normal hyperbolicity.
\end{rmk}

The following theorem of Burns--Wilkinson \cite{BWerg} shows that, in a very general setting, partial hyperbolicity together with accessibility implies ergodicity.

\begin{thm}[Theorem~1, \cite{BWerg}]\label{accerg}
Let $f$ be a volume-preserving, partially hyperbolic diffeomorphism that is center bunched and accessible. Then $f$ is ergodic.
\end{thm}

\subsection{Smoothness of $su$-holonomy}
We now consider the smoothness of the stable and unstable foliations in the transverse direction.

In this paper, we shall mainly consider stable and unstable holonomies with the center leaves as sections.

Suppose $f$ is a dynamically coherent, normally hyperbolic diffeomorphism with center foliation $\W^c$. Then, restricted to a center-stable leaf $\W^{cs}$, we may consider the local holonomy of the stable foliation from one center leaf to another. Let $p\in M$ and $q\in \W^s(p)$ be a pair of points in the same stable leaf of $f$. Then there exist sufficiently small neighborhoods $U(p)\subset \W^c_f(p)$ of $p$ and $U(q)\subset \W^c_f(q)$ of $q$ in the center leaves, and a constant $R=2d_{\W^s}(p,q)$, such that for every $x\in U(p)$, the set
\[
\W^s_f(x,R)=\{y\in \W^s_f(x): d_{\W^s}(x,y)<R\}
\]
has a unique intersection with $U(q)$. The local stable holonomy is then given by
\[
h^s_{f,p,q}: U(p)\to U(q), \quad x\mapsto y= U(q)\cap \W^s_f(x,R).
\]

The stable and unstable holonomies on the center leaves are known to be smooth.

\begin{thm}[Theorem~B, \cite{PSW}; Theorem~1.1, \cite{Saghin}]\label{c1hol}
Let $f\in \mathrm{Diff}^{r+1}(M)$, $r\ge 1$, be a normally hyperbolic, center $r$-bunched diffeomorphism on a Riemannian manifold $M$. Then the local stable and unstable holonomy maps for the center leaves are uniformly $C^r$. Furthermore, the holonomy $h^s_{f,p,q}$ from $\W^c_f(p)$ to $\W^c_f(q)$, where $q\in \W^s_f(p)$, depends $C^1$-continuously on $f$ with respect to the $C^1$ topology in a bounded neighborhood of $\mathrm{Diff}^{r+1}(M)$, and depends $C^r$-continuously on $p,q$ with respect to the $C^0$ topology on $M$.
\end{thm}

Instead of a single stable or unstable holonomy, we may also define \emph{holonomies of an $su$-path}. Given an $su$-path $\gamma$ of $f$ with endpoints $x_0,x_1,\ldots,x_k$, we may locally define the \emph{$su$-holonomy} in a neighborhood of $x_0$ in $\W^c(x_0)$ by concatenating stable and unstable holonomies between center leaves inside center-stable or center-unstable leaves following the path $\gamma$. More concretely, the local $su$-holonomy of a path $\gamma$ with endpoints $[x_0,x_1,\ldots,x_k]$, where $x_{i+1}\in \W^{t(i)}(x_i)$ for $0\le i\le k-1$ and $t(i)\in\{s,u\}$, is given by
\[
h_\gamma^f
= h^{t(k-1)}_{f,x_{k-1},x_k}\circ h^{t(k-2)}_{f,x_{k-2},x_{k-1}}
\circ \cdots \circ h^{t(0)}_{f,x_0,x_1}
: U(x_0)\to U(x_k),
\]
where $U(x_i)\subset \W^c_f(x_i)$ are sufficiently small neighborhoods of $x_i$ in its center leaf.

As a direct application of Theorem~\ref{c1hol}, we see that for partially hyperbolic diffeomorphisms satisfying the hypotheses of that theorem, the $su$-holonomies on center leaves are also smooth and depend smoothly on $f$ and $\gamma$.

We say that $f$ has \emph{global holonomy} on a center leaf $\W^c(x_0)$ if, for any $su$-path $\gamma$ of $f$ starting in that leaf and ending in $\W^c(x_k)$, the domain of $h_\gamma^f$ can be extended to the entire center leaf $\W^c(x_0)$. More precisely, $f$ has global holonomy on $\W^c(x_0)$ if for any $su$-path $\gamma=[x_0,x_1,\ldots,x_k]$ there exists a family of $su$-paths $\gamma_i$, $i\in I$, with endpoints in $\W^c(x_0)$ and $\W^c(x_k)$ such that
\begin{itemize}
\item the domains $U_i\subset \W^c(x_0)$ of the local holonomies $h_{\gamma_i}^f:U_i\to \W^c(x_k)$ cover $\W^c(x_0)$, i.e.\ $\bigcup_i U_i=\W^c(x_0)$;

\item the local holonomies of the $\gamma_i$ agree, i.e.\ $h_{\gamma_i}^f=h_{\gamma_j}^f$ on $U_i\cap U_j$ for any $i,j\in I$.
\end{itemize}

In the setting of Theorems \ref{cenrig} and \ref{cenrigsemi}, we show that $f$ has global holonomy in a generic center leaf that is simply connected.

\begin{prop}
 Let $f_0$ be a non-trivial element of the Weyl chamber flow, then any $C^1$-perturbation $f$ of $f_0$ has global holonomy in the universal cover. In particular, $f$ has global holonomy in a center leaf $\W^c_f(x_0)$ if $\W^c_f(x_0)$ is simply connected.
\end{prop}

\begin{proof}
 We lift the diffeomorphisms $f,f_0$ and the corresponding foliations in $G/\Gamma$ to the cover $ G$, denoted by $\tilde{f_0},\tilde f:G\to G$ and $\tilde{\mathcal W}^*_{f_0}$, $\tilde{\mathcal W}^*_f$, for $*=s,c,u,sc,cu$.

 Since $f_0$ is an affine diffeomorphism, $\tilde{\mathcal W}^*_{f_0}$ are exactly the orbits of $G^*=\exp(\mathfrak{g}^*)$ under left translation, where $\mathfrak{g}=\mathfrak{g}^s\oplus \mathfrak{g}^c\oplus \mathfrak{g}^u$ is the invariant splitting into generalized eigenspaces with eigenvalues $|\lambda|<1$, $|\lambda|=1$ and $|\lambda|>1$, given by the action of $Df_0$ on the Lie algebra $\mathfrak{g}$; and $\mathfrak{g}^{cs}=\mathfrak{g}^{c}\oplus \mathfrak{g}^{s},\mathfrak{g}^{cu}=\mathfrak{g}^{c}\oplus \mathfrak{g}^{u}$.
 
 We first claim that the foliations $\tilde\W^c_{f_0}$ and $\tilde\W^s_{f_0}$ have global product structure inside a leaf $\tilde \W^{cs}_{f_0}$. This is because we have a natural map $G^c\times G^s\to G^{cs},(g_c,g_s)\mapsto g_cg_s$ which is a diffeomorphism.

 This allows us to define a continuous map on each $cs$-leaf of $\tilde f_0$ given by $\tilde \W^{cs}_{f_0}(x)\to G^s: y\mapsto g_s(yx^{-1})$ where $g_s: G^{cs}\to G^s$ is just the projection under the product structure $g_cg_s\mapsto g_s$.

 Since $f_0$ is dynamically coherent, normally hyperbolic, center bunched with smooth center leaves, for sufficiently small $C^1$-perturbations $f$ of $f_0$, there exists a leaf conjugacy $\phi$ from $(f_0,\W^c_{f_0})$ to $(f,\W^c_{f})$. By the dynamics of $f$ and $f_0$, $\phi$ also sends $\W^{cs}_{f_0}$ to $\W^{cs}_{f}$. We take the map $u:\tilde\W^{cs}_f(x_0)\to G^s$ given by $u(x)=g_s(\phi^{-1}(x)\cdot (\phi^{-1}(x_0))^{-1})$. Then by construction, leaves of $\tilde\W^c_f$ are exactly the level sets of the map $u$.

 We claim that $u$ restricted to any stable leaf $L=\tilde\W^s_f(x_1)$ is a homeomorphism. 
 
 Firstly, we show that $u|_L$ is injective. If $u(x)=u(y)$ for $x,y\in \tilde \W^s_f(x_1)$, then $y\in \tilde \W^c_f(x)\cap \tilde \W^s_f(x)$. By the dynamics in the stable and center foliation, this shows that $d(f^n(x),f^n(y))\le C_1 \mu^n d(x,y)$ and $d_{\tilde \W^c_f}(f^n(x),f^n(y))\ge C_2 \nu^n d_{\tilde \W^c_f}(x,y)$ for any $n>0$, where $0<\mu<\nu<1, C_1,C_2$ are fixed constants given by the dominated splitting that depends only on $f_0$ and the $C^1$-neighborhood where we pick $f$. Since the leaf conjugacy $\phi^{-1}$ is $C^1$-close to identity along $\W^c_f$ and $C^0$-close to identity in $X$, we have $$d_{\tilde \W^c_f}(f^n(x),f^n(y))\le 2 d_{\tilde \W^c_{f_0}}(\phi^{-1} f^n(x),\phi^{-1} f^n(y))$$ and $$ d_{G}(\phi^{-1} f^n(x),\phi^{-1} f^n(y))\le d_{G}( f^n(x),f^n(y))+2\epsilon\le C_1\mu^nd(x,y)+2\epsilon\le 3\epsilon$$ for sufficiently large $n>N_\epsilon $, and $\epsilon$ can be taken to be arbitrarily small as $d_{C^1}(f,f_0)\to 0$. Since $d_{G}(\phi^{-1} f^n(x),\phi^{-1}f^n(y))\le 3\epsilon$ and $\phi^{-1} f^n(y)\in \tilde\W^c_{f_0}(\phi^{-1} f^n(x))$ the local geometry in $\tilde G$ gives a uniform bound $$C_3 d_{G}(\phi^{-1} f^n(x),\phi^{-1} f^n(y))\ge d_{\tilde \W^c_{f_0}}(\phi^{-1} f^n(x),\phi^{-1} f^n(y))$$ where $C_3>1$ depends only on $f_0$ and $G$.
 
 Therefore, we have $d_{\tilde \W^c_f}(f^n(x),f^n(y))<6C_3\epsilon $ for $n>N_\epsilon $. Pick a sufficiently small $\epsilon$, this implies that $d_{\tilde \W^c_f}(f^n(x),f^n(y)) $ is comparable to $d_X(f^n(x),f^n(y))$ for any $n>N_\epsilon $. On the other hand, we have $$d(f^n(x),f^n(y))\le C_1 \mu^n d(x,y)$$ and $$d_{\tilde \W^c_f}(f^n(x),f^n(y))\ge C_2 \nu^n d_{\tilde \W^c_f}(x,y)$$ for any $n\in \N$, since $0<\mu<\nu<1$, letting $n\to \infty$ gives a contradiction unless $x=y$. 
 
 Secondly, $u|_L$ is locally a homeomorphism, since local pieces of $\W^s_f$ are homeomorphic to pieces of $\W^s_{f_0}$ via projection in the $\W^c_f$ direction. Moreover, $u|_L$ is proper. Let $K$ be any compact set in $G^s$, then $u^{-1}(K)\cap L$ is a bounded set in $\tilde G$ since $\phi^{-1}$ is uniformly $C^0$-close to identity and $E^s_f$ is $C^0$-close to $E^s_{f_0}$. Therefore, $u(L)$ is both open and closed in $G^s$, so $u$ is also surjective.

 This shows that for any $x,y\in G$, $\tilde \W^s_f(x)\cap \tilde \W^c_f(y)=\{z\in \tilde \W^s_f(x): u(z)=u(y)\}$ is a unique point in $G$. Therefore, stable holonomies of $\tilde f$ are globally well-defined in $G$, and similarly, unstable holonomies of $\tilde f$ are globally well-defined in $G$. Passing to the quotient $G/\Gamma$, for any $su$-path $\gamma$, a lifting $\tilde \gamma$ can be globally extended to continuous family of paths from $\tilde \W^c(\tilde \gamma(0))$ to $\tilde \W^c(\tilde \gamma(1))$, since $\W^c_f(x_0)$ is simply connected, the projection $ G\to G/\Gamma$ restricted to $\tilde \W^c_f(x_0)$ is a homeomorphism, this continuous family of paths project to a continuous family of paths from $\W^c_f(x_0)$ to $\W^c_f(\gamma(1))$. This shows that $f$ has global holonomy on $\W^c_f(x_0)$.
 
\end{proof}

\subsection{Some normal form theory}\label{sec:normalform}
In this section, we introduce some normal form theory for uniformly contracting foliations, which we will use to promote the regularity of certain homeomorphisms in the centralizer and the regularity of the conjugacy map.

Let $f$ be a diffeomorphism on a closed manifold $M$ that uniformly contracts a foliation $\W^s$ whose leaves are uniformly $C^1$. Let $E^s$ be the tangent bundle of $\W^s$, and denote by $F:E^s\to E^s$ the restriction of $Df$ to the subbundle $E^s$. Let
\[
F^*:\Gamma(E^s)\to \Gamma(E^s), \qquad F^*(v)(x)=F(v(f^{-1}(x))),
\]
be the linear map between sections of $E^s$ induced by $F$. The spectrum of $F^*$ is called the \emph{Mather spectrum} of $F$; see \cite{Mather}. The spectrum consists of closed annuli in $\C$ centered at the origin and bounded by circles of radii $e^{\lambda_i}$ and $e^{\mu_i}$ satisfying
\[
\lambda_1 \le \mu_1 < \lambda_2 \le \mu_2 < \cdots < \lambda_l \le \mu_l < 0.
\]

\begin{defi}[Narrow band spectrum]
A bundle map $F$ is said to have a \emph{narrow band spectrum} if
\[
\mu_i+\mu_l \le \lambda_i,\qquad \forall\, 1\le i \le l.
\]
We call the number $s(F):=\frac{\lambda_1}{\mu_l}$ the \emph{critical regularity} of $F$.
\end{defi}

We now introduce the main theorem of this section, which we shall use to upgrade the regularity of certain homeomorphisms commuting with $f$.

\begin{thm}[\cite{normalform}]\label{normalform}
Let $f$ be a $C^r$ diffeomorphism that uniformly contracts an invariant foliation $\W^s$ with uniformly $C^r$ leaves. Suppose the bundle map $F=Df|_{T\W^s}$ has narrow band spectrum with $s(F)<r$. Then there exists a family of local diffeomorphisms
\[
H_x:\W^s(x)\to T\W^s(x)=E^s(x),
\]
such that:

\begin{enumerate}
\item $P_x=H_{f(x)}\circ f\circ H_x^{-1}:E^s(x)\to E^s(f(x))$ is a polynomial map of degree $\le s(F)$ for any $x\in M$;

\item $H_x(x)=0$ and $D_xH_x=\mathrm{Id}$ for any $x\in M$;

\item $H_x$ depends $C^r$-continuously on $x$ and is jointly $C^r$ in $x$ and $y\in \W^s(x)$;

\item if $g$ is a homeomorphism commuting with $f$, preserving $\W^s$, and is $C^r$ along the leaves of $\W^s$, then
\[
Q_x:=H_{g(x)}\circ g\circ H_x^{-1}
\]
is also a polynomial map of degree $\le s(F)$.
\end{enumerate}
\end{thm}

\subsection{Some Pesin theory}
We now recall the Oseledets splitting theorem and some basic Pesin theory that we will use in this paper. We first recall the Oseledets splitting theorem, which describes the asymptotic behavior of diffeomorphisms under iteration. Here we state the theorem in the setting of abelian group actions.

\begin{thm}[Oseledets Splitting Theorem]\label{Oseled}
Let $\alpha:A\to \mathrm{Diff}(M)$ be an abelian group action on a compact Riemannian manifold $M$ preserving an invariant Borel measure $\mu$.

Then for $\mu$-a.e.\ $x\in M$, there exists a decomposition
\[
T_xM=\bigoplus H_i(x)
\]
such that the decomposition is $D\alpha(a)$-invariant for any $a\in A$, and the \emph{Lyapunov exponents} $\chi_i:A\times M\to \R$, given by
\[
\chi_i(a)(x)=\lim_{k\to \infty} \frac{1}{|k|}\log\frac{\|D\alpha(a)^k v\|}{\|v\|},
\]
exist for any $v\in H_i(x)\setminus\{0\}$.
\end{thm}

\begin{rmk}
If we take a single diffeomorphism $f$ that generates an action $\Z\to \mathrm{Diff}(M)$, then we simply denote $\chi_i(1)(x)$ by $\chi_i(x)$, and we may further assume that
\[
\chi_1(x)<\chi_2(x)<\cdots<\chi_l(x)
\]
for almost every $x\in M$.

If $\alpha$ is $\mu$-ergodic, then the Lyapunov exponents are constant almost everywhere with respect to $x$, since they are $\alpha$-invariant, and we denote them by $\chi_i:A\to \R$. The Oseledets splitting theorem also holds in greater generality for cocycles other than the derivative; see \cite{introdyn}, Supplement~2.
\end{rmk}

A diffeomorphism $f$ with some non-zero Lyapunov exponents is said to be \emph{non-uniformly partially hyperbolic}. Similar to the case of partially hyperbolic diffeomorphisms, if a point $x\in M$ has Lyapunov exponent $\chi_k(x)<0$, then we may define stable and unstable manifolds for such diffeomorphisms.

\begin{thm}[Pesin Stable Manifold Theorem \cite{Pes}]\label{pesin}
Let $M$ be a compact Riemannian manifold and let $f:M\to M$ be a $C^2$ diffeomorphism with invariant measure $\mu$. Suppose the Lyapunov exponents satisfy
\[
\chi_k(x)<0\le \chi_{k+1}(x)
\]
for almost every $x\in M$. Then there exist Borel functions $\lambda:M\to(0,1)$ and $C:M\to(0,+\infty)$ such that for almost every $x\in M$, the Pesin stable manifold
\[
\W^s_f(x)=\{y\in M:\ d(f^n(x),f^n(y))\le C(x)\lambda(x)^n d(x,y)\ \forall n\ge 0\}
\]
is a $C^1$ submanifold of $M$ satisfying
\[
T_x\W^s_f(x)=\bigoplus_{i\le k} H_i(x).
\]
Furthermore, the stable manifolds satisfy $\W^s_f(x)=\W^s_f(y)$ for any $y\in \W^s_f(x)$.
\end{thm}

For an abelian action $\alpha$, we define the \emph{coarse Lyapunov foliations} of the action to be the maximal intersections of the Pesin stable manifolds of elements of $\alpha$.

\section{Center-fixing centralizer is a Lie group}\label{cenfix}

As we noted in Section~\ref{outl}, a general method for proving centralizer rigidity for partially hyperbolic diffeomorphisms is to separate the centralizer into the center-fixing part and the action on the leaf space of the center foliation via the following short exact sequence:
\[
1\to \CZ^r(f)\to \mathcal{Z}^r(f)\to \mathcal{Z}^r(f)/\CZ^r(f)\to 1.
\]

In this section, we deal with the center-fixing part and prove that it is a Lie group. We state and prove a general condition under which the centralizer of a partially hyperbolic diffeomorphism is a Lie group.

\begin{thm}\label{czlieg}
Let $M$ be a compact Riemannian manifold. Let $f$ be a partially hyperbolic $C^{r+1}$ diffeomorphism of $M$ with $r\ge 1$. Suppose that $f$ is accessible, dynamically coherent, center $r$-bunched, has global holonomy on a leaf $\W^c_f(x_0)$, and satisfies the narrow band spectrum condition (in both $E^s_f$ and $E^u_f$) with critical regularity $s(f)<r$. Then its $\W^c$-fixing $C^r$ centralizer $\CZ^r(f)$ is a Lie group acting freely and properly on the center leaf $\W^c_f(x_0)$, and we have
\[
\dim \CZ^r(f)\le \dim \W^c_f(x_0).
\]
\end{thm}

Theorem~\ref{czlieg} is a generalization of Theorem~3 in \cite{DWX23}, where the authors proved that for an accessible isometric extension $f_0$ of an Anosov diffeomorphism of a closed nilmanifold, the centralizer of a $C^1$ perturbation $f$ of $f_0$ is a Lie group acting freely on the compact center leaves. The main difference and difficulty here is that the center leaves may be non-compact. We postpone the proof to Section~\ref{pfczlieg}.

Applying Theorem~\ref{czlieg} to the setting of Theorem~\ref{cenrig}, we see that the center-fixing centralizer is a Lie group.

\begin{cor}
Let $f\in \mathrm{Diff}^{r+1}(G/\Gamma)$ be a perturbation of an element of the Weyl chamber flow as in Theorem~\ref{cenrig}, where $r> s(f)$. Then its center-fixing centralizer $\CZ^r(f)$ is a Lie group that acts properly and freely on a generic center leaf $\W^c_f(x_0)$. In particular, we have
\[
\dim(\CZ^r(f))\le k,
\]
where $k=\mathrm{rank}(G)$.
\end{cor}

\begin{proof}
We verify that $f$, as a $C^1$ perturbation of $f_0$, satisfies the conditions of Theorem~\ref{czlieg}.

Indeed, $f$ is accessible by Proposition~\ref{faccess}. By Theorem~\ref{leafconj}, $f$ is dynamically coherent. Moreover, $f$ is center $r$-bunched and satisfies the narrow band spectrum condition for any $r>s(f_0)$, since $f_0$ satisfies these conditions and these conditions are $C^1$-open.

We also note that although $f$ does not have global holonomy on every leaf of $\W^c_f$, $f$ has global holonomy when lifted to $G$, and thus has global holonomy on every generic center leaf of $\W^c_f$.

Therefore, applying Theorem~\ref{czlieg} yields the desired result.
\end{proof}

\subsection{Other applications}

We may also apply Theorem~\ref{czlieg} to the settings of \cite{DWX} and \cite{BG}. When applied to systems with $1$-dimensional center, since Lie groups acting freely and properly on $\R$ are either $\R$ or $\Z$, the center-fixing centralizer $\CZ^r(f)$ must be $\Z$ or $\R$. As a corollary, we may remove the volume-preserving assumptions in Theorem~1 of \cite{DWX} and Theorems~A, B, and~C in \cite{BG}, and obtain the following results.

\begin{thm}\label{1dcen1}
Let $\phi_t:T^1M\to T^1M$ be the geodesic flow on a closed, negatively curved, locally symmetric manifold $M$, and let $f\in \mathrm{Diff}^\infty(T^1M)$ be a $C^1$-small perturbation of $\phi_{t_0}$. Then $\mathcal{Z}^\infty(f)$ is virtually either $\Z$ or $\R$, i.e.\ $f$ either has virtually trivial centralizer or embeds into a smooth flow.
\end{thm}

\begin{thm}\label{1dcen2}
Let $f$ be a partially hyperbolic $C^\infty$ diffeomorphism on a compact $3$-manifold $M$. If $f$ is a discretized Anosov flow and $\pi_1(M)$ is not virtually solvable, then either $\mathcal{Z}^\infty(f)$ is virtually trivial or $f$ embeds into a smooth Anosov flow.
\end{thm}

We use the following result from \cite{DWX} and \cite{BG}, which shows that the centralizer is virtually center-fixing.

\begin{prop}[Proposition~37, \cite{DWX}; Proposition~2.4, \cite{BG}]
Let $f$ be a discretized Anosov flow satisfying the conditions of Theorem~\ref{1dcen1} or Theorem~\ref{1dcen2}. Then $\mathcal{Z}^r(f)/\CZ^r(f)$ is finite for $r\ge 1$.
\end{prop}

Therefore, applying Theorem~\ref{czlieg} yields Theorem~\ref{1dcen1} and Theorem~\ref{1dcen2}.

\begin{proof}[Proof of Theorem~\ref{1dcen1} and Theorem~\ref{1dcen2}]
For Theorem~\ref{1dcen1}, since $M$ is closed, negatively curved, and locally symmetric, $\phi_{t_0}$ has narrow band spectrum; see \cite[Lemma~26]{DWX}. Therefore, $f$ also has narrow band spectrum. Moreover, $f$ is dynamically coherent by Theorem~\ref{leafconj}, and it is center $r$-bunched since the center is $1$-dimensional. Lifting $f$ to the universal cover of $M$ shows that global holonomy exists on a generic geodesic.

For Theorem~\ref{1dcen2}, since $M$ is $3$-dimensional, $f$ satisfies the narrow band spectrum and the center bunching condition. Lifting $f$ to the universal cover of $M$ shows that global holonomy exists on a generic geodesic.

Thus, we may apply Theorem~\ref{czlieg} to conclude that $\CZ^r(f)$ is a Lie group acting freely and properly on $\R$, and hence must be $\Z$ or $\R$. Since $\mathcal{Z}^r(f)/\CZ^r(f)$ is finite, the centralizer $\mathcal{Z}^r(f)$ is virtually trivial or virtually $\R$.
\end{proof}
\subsection{Proof of Theorem \ref{czlieg}}\label{pfczlieg}
In this section, we prove Theorem \ref{czlieg}. We use ideas similar to those in Section 4 of \cite{DWX23}. {We shall construct a continuous center-fixing centralizer of $f$ and show that it is closed in $\mathrm{Homeo}(M)$ and acts freely and properly on $\W^c_f(x_0)$. Then we use the normal form theory to upgrade the smoothness of the continuous maps to $C^r$. Finally, applying the Hilbert-Smith conjecture in the case of diffeomorphisms, this shows that the smooth centralizer is a Lie group.}

Let $f$ be a normally hyperbolic diffeomorphism on a manifold $M$ with center foliation $\W^c_f$. Recall that we denote by
$$\mathcal{Z}^r(f)=\{g\in \mathrm{Diff}^r(M):\; f\circ g=g\circ f\}$$
the $C^r$ centralizer of $f$, and denote
$$\CZ^r(f)=\{g\in \mathrm{Diff}^r(M):\; f\circ g=g\circ f,\; g(\W^c_f(x))=\W^c_f(x),\; \forall x\in M\}$$
to be the $\W^c_f$-fixing, $C^r$ centralizer of $f$.

Following the notations in \cite{DWX23}, we denote by
$$\CZ_*(f)=\{g\in \mathrm{Homeo}(M): f\circ g=g\circ f; g(\W^c_f(x))=\W^c_f(x), \forall x\in M; g(\W^{*}_f)=\W^{*}_f, *=s,u\}$$
the continuous centralizer of $f$ which fixes each $\W^c_f$ leaf and preserves the stable and unstable foliations of $f$. Since $f$ and its stable, unstable, and center foliations are continuous, we see that $\CZ_*(f)$ is a closed subgroup of the homeomorphism group under the $C^0$ topology.

We denote by $\CZ^r_*(f)$ the elements of $\CZ_*(f)$ which are locally uniformly (with respect to different leaves) $C^r$-smooth along the leaves of $\W^c_f$.

Clearly, by definition, we have $\CZ^r(f)\subset \CZ^r_*(f)\subset \CZ_*(f)$.

An easy consequence of $\CZ_*(f)$ preserving the $s/u$-foliations and fixing the center leaves is that it must commute with any $su$-holonomy with respect to the center leaves.

\begin{prop}\label{comhol}
Let $g\in \CZ_*(f)$, {and let $x\in \W^c_f(x_0)$. Then, for any $su$-path $\gamma$ from $x$ to $z\in M$, we have
$$g(h_\gamma(y))=h_\gamma(g(y))$$
for any $y\in \W^c_f(x_0)$. Here, $h_\gamma:\W^c_f(x)\to \W^c_f(z)$ denotes the global $su$-holonomy given by the path $\gamma$.}
\end{prop}

Using this observation and the fact that $f$ is accessible and center $r$-bunched so that the $su$-holonomies act transitively $C^r$ on every center leaf, we may prove that every continuous map in $\CZ_*(f)$ is, in fact, $C^r$ smooth along the $\W^c_f$-leaves. 

\begin{lem}\label{CrWc}
Under the assumptions of Theorem \ref{czlieg}, we have $\CZ^r_*(f)=\CZ_*(f)$. 
\end{lem}

Furthermore, we may show that the group $\CZ_*(f)$ is locally compact, again using the fact that it commutes with the $su$-holonomies.

\begin{prop}\label{CZcpt}
Under the assumptions of Theorem \ref{czlieg}, the group $\CZ_*(f)$ {acts properly on the center leaf $\W^c_f(x_0)$.}
\end{prop}

We postpone the proofs of Lemma \ref{CrWc} and Proposition \ref{CZcpt} to Section \ref{pflem}.

Now we prove Theorem \ref{czlieg}.
\begin{proof}[Proof of Theorem \ref{czlieg}]
Since $f$ is accessible, the $su$-holonomies of $f$ act transitively on any $\W^c_f$-leaf. {Therefore, for any $x,y\in \W^c_f(x_0)$, there exists a global $su$-holonomy $h_{x\to y}:\W^c_f\to \W^c_f$ given by an $su$-path from $x$ to $y$. By Proposition \ref{comhol}, any element $g\in \CZ_*(f)$ must satisfy $g(y)=g(h_{x\to y}(x))=h_{x\to y}(g(x))$. This implies that if $g(x)=x$ for some $x\in \W^c_f(x_0)$, then $g(y)=y$ for any $y\in \W^c_f(x_0)$.} Therefore, by Lemma \ref{CrWc} and Proposition \ref{CZcpt}, $\CZ_*^r(f)$ {acts freely and properly} by diffeomorphisms on $\W^c_f(x_0)$. 

{We now show that the group $\CZ^r(f)\subset \CZ_*(f)$ is closed under the $C^0$ topology in $\mathrm{Homeo}(M)$.}

Suppose {$g_n\in \CZ^r(f)$, $g_n\to g\in \CZ_*(f)$.} Since $f$ satisfies the narrow band spectrum, by normal form theory {(see Theorem \ref{normalform})}, any $g_n\in \CZ^r(f)$ can be written as a polynomial map of bounded degree and coefficients under a locally uniform $C^\infty$ coordinate change; thus $g$ is also locally uniformly $C^\infty$ smooth along {$\W^s_f$} and $\W^u_f$. Moreover, $g$ is $C^r$ along $\W^c_f$ by Lemma \ref{CrWc}. Therefore, since $f$ is dynamically coherent, by Journ\'e's lemma \cite{Journe}, $g$ is $C^r$ smooth along $\W^{sc}_f$. Applying Journ\'e's lemma again, we see that $g$ is $C^r$ smooth on $M$.

Therefore, $\CZ^r(f)$ is a closed subgroup of $\CZ_*(f)$, hence it is a locally compact group acting freely and properly on $\W^c_f(x_0)$ by $C^r$ {diffeomorphisms. Therefore, by the Hilbert-Smith conjecture in the case of diffeomorphisms, see \cite[Theorem 2]{Hilsch}, $\CZ^r(f)$ is a Lie group acting freely on $\W^c_f(x_0)$. This gives us the desired result.}
\end{proof}

\subsection{Proof of local compactness}\label{pflem}

We first use the following proposition, which is a slight variant of Proposition 11 in \cite{DWX23}, to show that we can raise the regularity of $g\in \CZ_*(f)$ to $C^r$ on the center leaves.

\begin{prop}
Let {$g:N\to N$} be a homeomorphism on a $C^r$ Riemannian manifold {$N$}, where $1\le r\le \infty$. If $\mathcal{Z}^r(g)$ acts transitively on {$N$}, then $g$ is a $C^r$ diffeomorphism of $N$.
\end{prop}

\begin{proof}
Take the image manifold $\Gamma(g)=\{(x,g(x))\in N\times N:x\in N\}\subset N\times N$. Since $\mathcal{Z}^r(g)$ acts transitively on $N$, it induces a transitive action by $C^r$ maps on $\Gamma(g)$ via $\{\phi\times \phi:\phi\in \mathcal{Z}^r(g)\}$. Thus, the manifold $\Gamma(g)$ is a locally compact $C^r$-homogeneous manifold. Therefore, by Theorem B of \cite{wilkcoh}, the manifold $\Gamma(g)$ is $C^r$.

Take the projections $p_1,p_2$ from $N\times N\to N$. They are clearly $C^r$ homeomorphisms from $\Gamma(g)$ to $N$. Using Sard's Theorem, we may pick a point $x\in\Gamma(g)$ such that $p_1$ has a $C^r$ inverse in a neighborhood $U_x\subset \Gamma(g)$ of $x$. Therefore, $g=p_2\circ p_1^{-1}$ is a $C^r$ map on $U_x$. Now, for any other $y\in N$, picking $h\in \mathcal{Z}^r(g)$ with $h(x)=y$, we see that $g\mid_{h(U)}=h\circ g\mid_U \circ h^{-1}$ is also $C^r$. Therefore, $g$ is $C^r$ at any $y\in N$, hence it is a $C^r$ diffeomorphism of $N$.
\end{proof}

Applying the proposition to $N=\W^c_f(x_0)$ and $g$ restricted to $\W^c_f(x_0)$, we may prove Lemma \ref{CrWc}.

\begin{proof}[Proof of Lemma \ref{CrWc}]
By $r$-bunching of $f$, the global $su$-holonomies of $f$ from $\W^c_f(x_0)$ to itself {are} $C^r$. Since $f$ is accessible, the global $su$-holonomies act transitively on $\W^c_f(x_0)$. Therefore, for any $g\in \CZ_*(f)$, $\mathcal{Z}^r(g)$ acts transitively on $\W^c_f(x_0)$ since it contains all the $su$-holonomies of $f$ from $\W^c_f(x_0)$ to itself.

Furthermore, $g$ is uniformly $C^r$ between nearby leaves since it commutes with holonomies between $\W^c_f$ leaves, and the holonomies are $C^r$-continuous with respect to the base points by Theorem \ref{c1hol}.

Therefore, we have $\CZ_*(f)=\CZ_*^r(f)$. This proves Lemma \ref{CrWc}.
\end{proof}

Now we proceed to prove Proposition \ref{CZcpt}.

Take a compact subset $K\subset \W^c_f(x_0)$. Suppose $g_n\in \CZ_*(f)$ such that $g_n(K)\cap K\neq \varnothing$ for all $n\in \N$. Picking a subsequence if necessary, we may assume that for some $x_n\in K$, $g_n(x_n)\in K$ and $x_n\to x$ as $n\to \infty$. By Proposition \ref{comhol}, this implies that $g_n(x)=h_{x_n\to x}(g_n(x_n))$ for any global $su$-holonomies from $x_n$ to $x$. Since $x_n,x,g_n(x_n)\in K$, we may choose the holonomies to be of bounded length and legs. This shows that $g_n(x)$ lies in a compact subset of $\W^c_f(x_0)$.

Therefore, by passing to a subsequence, we have $g_n(x)\to x_*$ for some $x_*\in \W^c_f(x_0)$. Using Proposition \ref{comhol} again, this shows that for any global $su$-holonomy $h_{x\to z}$ connecting $x$ and $z\in M$, we have
$$g_n(z)=g_n(h_{x\to z}(x))=h_{x\to z}(g_n(x))\to h_{x\to z}(x_*)$$
as $n\to \infty$.

Therefore, for any $z\in M$, the limit $\lim_{n\to \infty} g_n(z)$ exists and is equal to
$$\lim_{n\to \infty} g_n(z)=h_{x\to z}(x_*)$$
for any global $su$-holonomy $h_{x\to z}$ from $x$ to $z$.

To prove Proposition \ref{CZcpt}, we only need to show that the map $M \to M: z\mapsto h_{x\to z}(x_*)$ is continuous: if $g_n$ converges to a continuous map $g$, then since $f$ and the stable, unstable, and center foliations are continuous, $g$ also commutes with $f$, fixes $\W^c_f$, and preserves $\W^s_f$ and $\W^u_f$, hence $g\in \CZ_*(f)$.

The idea is first to show that one $su$-holonomy can be approximated by nearby $su$-holonomies that cover an open neighborhood, and then use accessibility of $f$ to carry this property to the whole manifold.

First, we define the following notion for two $su$-paths to be close.

\begin{defi}[$\epsilon$-near $su$-path]
We say that two $su$-paths $\gamma_1=[x_0,x_1,...,x_k]$ and $\gamma_2=[y_0,y_1,...,y_k]$ are $\epsilon$-near if each leg is $\epsilon$-near.

More precisely, $\gamma_1$ and $\gamma_2$ are $\epsilon$-near if $d_{\W^c}(x_i,y_i)<\epsilon$ and
$$d_M([x_i,x_{i+1}],[y_i,y_{i+1}])<\epsilon,$$
where $[x_i,x_{i+1}]$ is the $s/u$-segment in $\gamma_1$ starting at $x_i$ and ending at $x_{i+1}$, and $[y_i,y_{i+1}]$ is defined similarly.
\end{defi}

It is clear by definition and by the compactness of $M$ that near-enough $su$-paths give $su$-holonomies that are $C^0$-close.

\begin{lem}\label{sucont}
Fix $x_0\in M$. For every $\epsilon>0$, $R>0$, $k\in \N$, there exists $\delta>0$ such that for any $\delta$-near $su$-paths $\gamma_1$ and $\gamma_2$ with $\le k$ legs and lengths $\le R$ from $x_1\in \W^c(x_0)$, $d_{\W^c}(x_0,x_1)\le R$ to $y_1\in \W^c(z)$ and $x_2\in \W^c(x_0)$ to $y_2\in \W^c(z)$ respectively, the global $su$-holonomies $h_1$ and $h_2$ along $\gamma_1$ and $\gamma_2$ satisfy
$$d_{\W^c}(h_1(x_0),h_2(x_0)) <\epsilon.$$
\end{lem}

Furthermore, we have the following result on ``continuity" of $su$-paths.

\begin{lem}[Local continuity of holonomy, see Proposition 7.3 \cite{ASV}]\label{ctshol}
Let $f$ be a partially hyperbolic diffeomorphism on a complete Riemannian manifold $M$. Suppose $f$ is accessible. Then for any $x\in M$, there exist $y\in M$ and a $su$-path $\gamma$ from $x$ to $y$ such that for any $\epsilon>0$, there exists $\delta>0$ so that for any $z\in B(y,\delta)$, there exists a $su$-path from $x$ to $z$ that is $\epsilon$-near $\gamma$.
\end{lem}

The lemmas use the same proof as Proposition 7.3 in \cite{ASV}, which employs a Baire category argument, although the formulations there are slightly different.

Now we are ready to prove Proposition \ref{CZcpt} on the local compactness of $\CZ_*(f)$ in $\mathrm{Homeo}(M)$.

\begin{proof}[Proof of Proposition \ref{CZcpt}]
{Let $K$ be a compact subset of $\W^c_f(x_0)$. Let $g_n\in \CZ_*(f)$ be a sequence such that $g_n(K)\cap K\neq \varnothing$. By passing to a subsequence, we assume $g_n(x)\to x_*$ for some $x,x_*\in \W^c_f(x_0)$. By the analysis above, for any $z\in M$, $g_n(z)$ converges to $h_{x\to z}(x_*)$ for any $su$-holonomy from $x$ to $z$. We first prove that the map $g: z\mapsto h_{x\to z}(x_*)$} is continuous.

Pick $y$ and $\gamma$ as in Lemma \ref{ctshol}. We claim that $z\mapsto h_{x\to z}(x_*)$ is continuous at $y$. Indeed, pick $R=\max\{l_{su}(\gamma),d_{\W^c}(x_0,y)\}$, where $l_{su}(\gamma)$ denotes the $su$-length of $\gamma$. By Lemma \ref{sucont}, for any $\epsilon>0$, there exists $\delta_1>0$ such that for any $\gamma^\prime$ from $x$ to $z\in \W^c$ that is $\delta_1$-near $\gamma$, we have {$d_{\W^c}(h_{x\to y}(x_*),h_{x\to z}(x_*))<\epsilon$}. Furthermore, from Lemma \ref{ctshol}, there exists $\delta>0$ such that for any $z\in B(y,\delta)$, there exists a $su$-path from $x$ to $z$ that is $\delta_1$-near $\gamma$. Therefore, for any $z\in B(y,\delta)\cap \W^c(x)$, we have {$d_{\W^c}(h_{x\to y}(x_*),h_{x\to z}(x_*))<\epsilon$}. This shows that $h_{x\to z}(x_*)$ is continuous at $z=y$.

Now, for any other point $y^\prime\in M$, take a $su$-holonomy $h_{y\to y^\prime}$ that takes $y$ to $y^\prime$. Then $h_{y\to y^\prime}$ is a diffeomorphism on $\W^c$. Thus, it takes the neighborhood $B(y)$ to an open neighborhood $B(y^\prime)$. For any $z^\prime\in B(y^\prime)$, take $z$ to be its preimage in $B(y)$ under $h_{y\to y^\prime}$. Then we have
$$h_{x\to z^\prime}(x_*)=h_{z\to z^\prime}h_{x\to z}(x_*)=h_{y\to y^\prime}h_{x\to z}(x_*).$$
The map $h_{y\to y^\prime}h_{x\to z}(x_*)$ is continuous at $y^\prime$ with respect to $z^\prime$ since $z^\prime \mapsto z$ is smooth and $z\mapsto h_{x\to z}(x_*)$ is continuous at $y$. Thus $h_{x\to z}(x_*)$ is also continuous at any $z=y^\prime\in M$. Therefore, $g_n$ converges to a continuous map $g$ if the sequence $g_n(x_*)$ converges.

{Since $f$ and its stable, unstable, and center foliations are continuous, we see that $g$ commutes with $f$, preserves the stable and unstable foliations, and fixes the center foliation. Passing to a subsequence if necessary, we may apply the same argument to $g_n^{-1}$, which must converge to another continuous map $g_1$ on $M$ showing $g_1\circ g=\mathrm{id}$. This shows that $g$ is a homeomorphism on $M$. In conclusion, we have a subsequence of $g_n$ converging to $g\in \CZ_*(f)$. This shows that the action of $\CZ_*(f)$ on $\W^c_f(x_0)$ is proper.}
\end{proof}

\section{The centralizer virtually fixes the center leaves}\label{centgp}

In this section, we consider $M=X=G/\Gamma$, and let $f$ be a perturbation of an element of the Weyl chamber flow as in Theorem \ref{cenrig}. We directly apply the following result of Witte \cite{Witte} to show that $\CZ^r(f)$ is a finite-index subgroup of $\mathcal{Z}^r(f)$.

\begin{thm}[See Theorem 1.1 of \cite{Witte}]\label{witte}
Let $G$ be a connected semisimple Lie group with finite center and no compact factor, and let $\Gamma$ be a lattice in $G$ such that $G$ has no normal subgroup $N\simeq PSL(2,\R)$ with $N\Gamma$ closed in $G$. Let $D$ be a connected Lie subgroup of $G$, and let $\mathcal{F}$ be the foliation of $G/\Gamma$ by cosets of $D$. {Suppose $\mathcal{F}$ has a dense leaf.}

Then for any homeomorphism $f: G/\Gamma\to G/\Gamma$ that preserves the leaves of $\mathcal{F}$, if the induced homomorphism $f_*\in \mathrm{Out}(\Gamma)$ is trivial, then $f$ must be a homeomorphism $G/\Gamma\to G/\Gamma$ that fixes every leaf of $\mathcal{F}$.
\end{thm}

\begin{rmk}
In \cite{Witte}, the theorem was stated for a semisimple Lie group $G$ with trivial center instead of finite center, and the conclusion was $f$ would be a composition of an affine map and a homeomorphism that fixes every leaf of $\mathcal{F}$. The proof in \cite{Witte} only used the trivial center assumption to apply Mostow-Margulis rigidity theorem, and to use the affine map to make $f_*$ trivial. Here, we simply add the assumption that $f_*:\Gamma\to\Gamma$ is trivial.
\end{rmk}

Applying this to our setting, and using the leaf conjugacy to bring the diffeomorphisms $g\in \mathcal{Z}^r(f)$ to homeomorphisms preserving cosets of $D$, we see that $\CZ^r(f)$ is a finite-index subgroup of $\mathcal{Z}^r(f)$.

\begin{cor}
Let $f_0,f,X$ be as in Theorems \ref{cenrig} or \ref{cenrigsemi}. Then we have
$$|\mathcal{Z}^r(f)/\CZ^r(f)|<C,$$
where $C$ is a constant depending only on $X$.
\end{cor}

\begin{proof}
Take $\phi: G/\Gamma\to G/\Gamma$ to be the leaf conjugacy from $f_0$ to $f$ (given in Theorem \ref{leafconj}) that sends the center foliations of $f_0$ to those of $f$. Since $g\in \mathcal{Z}^r(f)$ preserves the center foliation of $f$, we see that $\phi^{-1} g\phi$ is a homeomorphism that preserves the center foliations $\W^c_{f_0}$, which are cosets of the Cartan subgroup $D$. We note that generic leaves of the cosets of $D$ are dense in $G/\Gamma$.

Therefore, if $g_*$ is trivial in $\mathrm{Out}(\Gamma)$, then $(\phi^{-1} g\phi)_*$ is also trivial in $\mathrm{Out}(\Gamma)$ and by Theorem \ref{witte}, $\phi^{-1} g\phi$ fixes the cosets of $D$ and $g$ fixes the leaves of $\W^c_f$. This shows that the kernel of $\mathcal Z^r(f)\to \mathrm{Out}(\Gamma), g\mapsto g_*$ is in $\CZ^r(f)$. Moreover, by our assumption that $\Gamma$ is a lattice in a semisimple Lie group with simple factors of rank at least two, $\mathrm{Out}(\Gamma)$ is finite. This shows that $|\mathcal{Z}^r(f)/\CZ^r(f)|<C$ for a constant $C$ depending only on $X$. 
\end{proof}

\section{Dichotomy of the centralizer}\label{dichcen}

In this section, we prove Theorem \ref{cenrig} on the dichotomy of the centralizer. 

The scheme of the proof is as follows:

Step 1: We use some estimates on the holonomies to show that if $g$ is not too far from the identity at one point, then $g$ is very close to a translation by some vector at every point.

Step 2: We show that if every $g\in \CZ^r(f)$ is close to the direction of $f$, then the dimension of $\CZ^r(f)$ is bounded above. 

Step 3: In Sections \ref{eph} and \ref{wcf}, we prove that if some $g\in \CZ^r(f)$ is far from the direction of $f$, then $f$ is conjugate to an element of the Weyl chamber flow.

\subsection{Centralizer acts almost by translations on $\W^c$}

Let $f$ be a $C^1$-small perturbation of a generic element $f_0$ of the Weyl chamber flow on $X=G/\Gamma$. 

Let $\phi$ be the leaf conjugacy from $(X,\W^c_{f_0})$ to $(X,\W^c_f)$, as given in Theorem \ref{leafconj}. For any $g\in \CZ^r(f)$, we denote by $\hat{g}:=\phi^{-1} g\phi$ the conjugate of $g$ by the leaf conjugacy. Since $g$ fixes the center leaves of $f$, we see that $\hat{g}$ fixes the center leaves of $f_0$, which are the orbits of $D$. 

We first note that, since $\mathrm{Out}(\Gamma)$ is finite, we may restrict to the elements $g\in \CZ^r(f)$, with lifting $\tilde{g}(y\gamma)=\tilde g(y)\gamma$ for any $y\in G, \gamma\in \Gamma$.

From now on, we consider elements of the finite index subgroup $\CZ^r_0(f)=\{g\in \CZ^r(f): g_*=id\in \mathrm{Out}(\Gamma)\}$ of $\CZ^r(f)$, which is also a finite index subgroup of $\mathcal Z^r(f)$ by our discussion in Section \ref{centgp}. 

For $\hat{g}$ fixing the center leaves of $f_0$, we may lift it to a center-fixing diffeomorphism $\tilde{g}: G\to G$, and the element $\tilde{g}(y)y^{-1}\in D$ does not depend on the choice of $y\in G$ for $x=y\Gamma$. We also denote by $\tilde \W$ in $G$ the lifting of a foliation $\W$ in $X$.

We now show that there is a uniform estimate on the behavior of $\hat{g}$ at every $x\in X$.

The precise statement is as follows: 

\begin{prop}\label{z2:dtof0}
For any $\epsilon,C>0$, there exists $\delta>0$, such that for any $f\in \mathrm{Diff}^2(X)$ with $d_{C^2}(f,f_0)<C$ and $d_{C^1}(f,f_0)<\delta$, we have for any $x\in G$:
$$
d_{\tilde \W^c_{f_0}}(\tilde {g}(y),g_0(y))\le \epsilon e^{C_Gd_{\tilde \W^c_{f_0}}(x, \tilde {g}(x))} d_{\tilde \W^c_{f_0}}(x, \tilde {g}(x))
$$
for any $y\in G$ and any $g\in \CZ_0^r(f)$. Here, $\hat{g}$ is the conjugate of $g$ by the leaf conjugacy, and $g_0$ is an element of the Weyl chamber flow given by $g_0(y)=(\tilde{g}(x)x^{-1})\cdot y$, which depends on the choice of $x$; and $C_G$ is a constant that only depends on $G$.
\end{prop}

The idea of the proof is to exploit the fact that $g\in \CZ^r(f)$ commutes with global $su$-holonomies from $\tilde \W^c_f(x)$ to $\tilde \W^c_f(y)$. We postpone the proof to the end of this section.

For any $g\in \CZ^r_0(f)$, we consider the following notion for $g$ to be close to a translation. \emph{This is a cone-like condition where the $C^0$ closeness depends on the translation length.}

\begin{defi}
For any $g\in \CZ^r_0(f)$, we say $g$ is \emph{projectively $\epsilon$-close to a translation} if there exists an element of the Weyl chamber flow $g_0$ such that for any $y\in G$, 
\begin{equation}\label{ntr}
d_{\tilde \W^c_{f_0}}(\tilde {g}(y),g_0(y))\le \epsilon d_{\tilde \W^c_{f_0}}(g_0(y),y).
\end{equation} 
\end{defi}

Since $f$ is a $C^1$ perturbation of $f_0$, it is easy to see that $f$ is projectively $\epsilon$-close to $f_0$. We shall show that many of $g\in \CZ(f)$ are also projectively $\epsilon$-close to translations.

First we make the following general observation. The proof is an elementary exercise.

\begin{lem}\label{transadd}
If $g$ and $g^\prime$ are projectively $\epsilon$-close to translations $g_0$ and $g_0^\prime$, respectively, and if for some $0<c\le 1$,
$$
d(g_0\circ g_0^\prime(y),y) \ge c \big(d(g_0(y),y)+d(g_0^\prime(y),y)\big),
$$
then $g\circ g^\prime$ is projectively $\frac{\epsilon}{c}$-close to a translation.

In particular, if $g$ is projectively $\epsilon$-close to a translation $g_0$, then $g^{k}$ is projectively $\epsilon$-close to the translation $g_0^k$ for any $k\in \N$.
\end{lem}

For $g\in \CZ_0^r(f)$, as a direct corollary of Proposition \ref{z2:dtof0}, we may show that $g$ is either uniformly very far from the identity or projectively $\epsilon$-close to a translation.

\begin{cor}\label{trorf}
For any $R,C>0$, $0<\epsilon<1$, there exists $\delta>0$ such that for any $C^1$-perturbation $f$ of $f_0$ as in Proposition \ref{z2:dtof0} with $d_{C^2}(f,f_0)<C$ and $d_{C^1}(f,f_0)<\delta$, and for any $g\in \CZ^r_0(f)$, if $\inf_{x\in G} d_{\tilde \W^c_{f_0}}(\tilde g(x),x)\le R$, then $g$ is projectively $\epsilon$-close to a translation $g_0$.
\end{cor}

The proof is a direct application of Proposition \ref{z2:dtof0}.

\subsection{Dichotomy of the centralizer}

Combining the fact that many $g\in \CZ^r_0(f)$ ``act like translations" with the fact that $\CZ^r_0(f)$ is a Lie group, we can obtain our main results.

We discuss whether $g$ and $f$ are approximately in the same direction to derive our theorems. Roughly speaking, if they are, then the dimension of $\CZ^r_0(f)$ is bounded above; if not, then there exists a partially hyperbolic $\Z^2$-action on $X$ and $f$ must be conjugate to a diagonal action.

For any $\epsilon>0$, we say that a vector $v\in \mathfrak{h}_i$ in the Cartan subalgebra of a simple component $\mathfrak g_i$ is \emph{outside the $\epsilon$-cone of $u\in \mathfrak{h}_i$} if
\[
d_{\mathfrak{h}_i}(v, \R u) > \epsilon\, d_{\mathfrak{h}_i}(v,0).
\]
A vector $v$ is outside the $\epsilon$-cone of $u$ if left multiplication by $\exp(v)$ is not projectively $\epsilon$-close to $\exp(tu)$ for any $t$.

We say that a translation $g_0$ is \emph{outside the $\epsilon$-cone of $f_0$ in a simple component} $\mathfrak{g}_i$ if the projection of the translation vector
\[
\rho(g_0):=\log(g_0(x)x^{-1})
\]
onto $\mathfrak{h}_i$ is outside the $\epsilon$-cone of the projection of $\rho(f_0):=\log(f_0(x)x^{-1})$.

\begin{defi} 
We say that $g\in \CZ^r_0(f)$ \emph{is in a different direction from $f$} if $g$ is projectively $0.01$-close to a translation $g_0$ whose projection to every simple component is outside the $0.1$-cone of $f_0$, and $\|\rho(f_0)\|\le \|\rho(g_0)\|\le 2\|\rho(f_0)\|$. 
\end{defi}

The existence of elements in a different direction from $f$ guarantees that $f$ is smoothly conjugate to an affine action.

\begin{prop}\label{gph}
Let $f$ be as in Theorems \ref{cenrig} or \ref{cenrigsemi}. If there exists $g\in \CZ^r_0(f)$ in a different direction from $f$, then $f$ is $C^r$-conjugate to an element of the Weyl chamber flow. In particular, $\CZ^r(f)$ is virtually $\R^k$, where $k = \mathrm{rank}(G)$.
\end{prop}

The proof of Proposition \ref{gph} is postponed to Sections \ref{eph} and \ref{wcf}. In Section \ref{eph}, we show that if some $g\in \CZ^r_0(f)$ is in a different direction, then the action
\[
\Z^2 \to \mathrm{Diff}^r(X): (k,l) \mapsto f^k g^l
\]
is topologically close to an affine action, with the same type of contraction and expansion along appropriate foliations. In Section \ref{wcf}, we show that such an action must be smoothly conjugate to a restriction of the Weyl chamber flow.

Combining the results above, $f$ is conjugate to an element of the Weyl chamber flow unless every $g\in \CZ^r_0(f)$ is in the same direction as $f$ in some simple component. We now show that this imposes a dimension bound on $\CZ^r_0(f)$.

\begin{prop}\label{conerig}
Let $f:X\to X=G/\Gamma$ be a small perturbation of an element of the Weyl chamber flow as in Theorem \ref{cenrigsemi}. If there does not exist $g\in \CZ^r_0(f)$ in a different direction from $f$, then $\CZ^r_0(f)$ is a Lie group of dimension at most $k-\min_i k_i+1$, where $k=\mathrm{rank}\; G$ and $k_i$ are the ranks of the simple components of $G$. 
\end{prop}

\begin{proof}

Fix a generic center leaf $\W^c_{f_0}(x_0)$ and identify it with the Cartan subalgebra $\mathfrak{h}$ via the logarithmic map at $x_0$.

 Consider the map\[
\Psi: \CZ^r_0(f) \to \mathfrak{h}, \quad g \mapsto \log(\hat{g}(x_0)x_0^{-1}).
\]
By Proposition \ref{CZcpt}, $\CZ^r_0(f)$ is a Lie group acting freely and properly on center leaves by $C^r$ diffeomorphisms. Therefore, by Theorem 1 of \cite{BM}, the action $\CZ^r_0(f)\times X\to X$ is jointly $C^r$. Using the fact that the leaf conjugacy is smooth restricted to the center, this shows that the image $\mathrm{Im}(\Psi)$ is a smooth embedded submanifold of $\mathfrak{h}\simeq \R^k$.

Let
\[
K_i = \{y\in \mathfrak{h}_i: d(y, \pi_i\rho(f_0)\R) < 0.1 |y|\}.
\]

By our assumption, any $g\in \CZ^r_0(f)$ cannot be $0.01$-close to a translation $g_0$ not in the same direction as $f_0$.

Applying Proposition \ref{z2:dtof0} and Corollary \ref{trorf}, we have that any $g\in \CZ^r_0(f)$ with $\|\rho(f_0)\|\le \|\Psi(g)\|\le 2\|\rho(f_0)\|$ is projectively $0.1$-close to a translation $g_0(y)=\Psi(g)\cdot y$. By our assumption, the projection of $g_0$ to some simple component $\mathfrak{g}_i$ lies inside the $0.1$-cone of the corresponding projection of $f_0$. Moreover, if $\|\Psi(g)\|\le 1$ is projectively $0.01$-close to a translation $g_0$ outside the $0.1$-cone of $f_0$ in each simple component, then there exists $k\in \N$, such that $\|\rho(g_0^k)\|=k\|\rho(g_0)\|\in [\|\rho(f_0)\|,2\|\rho(f_0)\|]$. Therefore the power $g^k$ is projectively $0.01$-close to $g_0^k$, which is outside the $0.1$-cone in each simple component with $\|\rho(f_0)\|\le \|\Psi(g^k)\|\le 2\|\rho(f_0)\|$. This shows that for any $g\in \CZ^r_0(f)$ with $ \|\Psi(g)\|\le 2\|\rho(f_0)\|$, we must have $\Psi(g)\in \cup \pi_i^{-1}(K_i)$.

Therefore, we have by Corollary \ref{trorf},
\[
\mathrm{Im}(\Psi) \cap B(0,2\|\rho(f_0)\|) \subset \cup_i \pi_i^{-1}(K_i)=\cup_i (K_i\times\prod_{j\neq i}\mathfrak h_j)
.\]
 Let $C_i=\{0\}\cup(\mathfrak h_i-K_i)\subset \mathfrak h_i$ be the complementary cone. Let $P$ be the subspace tangent to $\mathrm{Im}(\Psi)$ in $\mathfrak h$ at $0$. Then we have $P\cap \prod_i(C_i-\{0\})=\varnothing$. Therefore, $P\cap \prod_i(C_i)\subset \cup_{i}\{0\}\times \prod_{j\neq i} C_j$. Let $P_i=(\pi_i\rho(f_0))^\perp\subset \mathfrak h_i$ be a maximal subspace inside $C_i$, then we have $\dim P_i=k_i-1$ and $\dim \prod_i(P_i)=\sum_i(\dim \mathfrak h_i-1)$. By $P\cap \prod_i(P_i)\subset \cup_{i}\{0\}\times \prod_{j\neq i} P_j$, we have the dimension count: 
$$\dim P+\sum_i(\dim \mathfrak h_i-1)-\dim \mathfrak h\le\dim (P\cap \prod_i(P_i))\le \max_j \sum_{i\neq j}(\dim \mathfrak h_i-1).$$ This shows that $\dim P\le \dim \mathfrak h-\min_i(\dim \mathfrak h_i-1)=k-\min k_i+1$

Since $\Psi$ is an embedding, this shows that $\dim \CZ^r_0(f) \le k - \min k_i + 1$. In particular, in the case when $G$ is simple, $\dim \CZ^r_0(f)=0$ or $1$.
\end{proof}

Theorem \ref{cenrig} follows directly from Propositions \ref{gph} and \ref{conerig}.

\subsection{Proof of Proposition \ref{z2:dtof0}}

We now return to the proof of Proposition \ref{z2:dtof0}.

By Proposition \ref{comhol}, for any center-fixing centralizer $g\in \CZ^r_0(f)$, we have
$$
g(y) = h^f_{x\to y}(g(x))
$$
for any $x,y\in X$ and any $f$-$su$-holonomy $h^f_{x\to y}: \W^c(x)\to \W^c(y)$ along a path from $x$ to $y$.

For any two points $x,y\in X$, take an $f$-$su$-holonomy from $\phi(x)$ to $\phi(y)$ and denote
$$
\hat{h}^f_{x\to y} := \phi^{-1} \circ h^f_{\phi(x)\to \phi(y)} \circ \phi
$$
to be the $su$-holonomy of $f$ pulled back to the center leaves of $f_0$ via the leaf conjugacy. We denote by $\tilde h^f_{x\to y}$ its lifting to $G$. By the above condition, $g(y)=h^f_{x\to y}(g(x))$, we then have
$$
\tilde{g}(y) = \tilde{h}^f_{x\to y}(\tilde{g}(x))
$$
for any $x,y\in X$.

Fix $x\in X$ and define
$$
g_0(y) = h^{f_0}_{x\to y}(\hat{g}(x)),
$$
a diffeomorphism $\W^c_{f_0}(x)\to \W^c_{f_0}(y)$ given by the $su$-holonomies of $f_0$. This map does not depend on the path and is an element of the Weyl chamber flow on $X=G/\Gamma$, acting by translation on the leaves of $\W^c_{f_0}$.

Intuitively, the difference between $\hat{g}$ and $g_0$ is the difference between the $su$-holonomies $\hat{h}^f$ and $h^{f_0}$, which should be small if $f$ is a small perturbation of $f_0$. We aim to show that $\hat{g}$ is $C^0$-close to the translation $g_0$ on the leaf $\W^c_{f_0}(x)$ unless $\hat{g}$ is far from the identity.

\begin{proof}[Proof of Proposition \ref{z2:dtof0}]
We consider two cases:

\begin{itemize}
\item{Case 1}: $|\tilde{g}(x)-x|\le 1$.

We consider the map $\tilde{h}^f_{x\to y}$, and identify $\tilde \W^c_{f_0}(x)$ and $\tilde \W^c_{f_0}(y)$ with $\mathfrak{h}\simeq \R^k$, for any $z\in \tilde \W^c_{f_0}(x)$, $$ \tilde{h}^f_{x\to y}(z)=\tilde{h}^f_{x\to y}(x)+\int_{\gamma} D_{\dot{\gamma}}(\tilde{h}^f_{x\to y})dt $$ $$=y+\int_{\gamma}D_{\dot{\gamma}}(\tilde{h}^f_{x\to y})dt $$ $$=g_0(y)+ \int(D_{\dot{\gamma}}(\tilde{h}^f_{x\to y})\mid_{x^\prime}-\mathrm{Id})dt$$ for any curve $\gamma(t)$ connecting $x$ and $z$.

Now plug in $z=\tilde{g}(x)$. We may always choose an $f$-$su$-path from $\phi(x)\to \phi(y)$ of $\le k$ legs and length $\le R$. Since $|x-z|<1$, the $f$-$su$-path given by the holonomy $h^f_{\phi(x)\to \phi(y)}$ starting at the point $x^\prime\in \gamma$ must also be of $\le k$ legs and of length $\le e^2 R$. Thus by Theorem \ref{c1hol}, there exists $\delta>0$, such that the operator norm $\|D(\tilde{h}^f_{x\to y})\mid_{x^\prime}-D(h^{f_0}_{x\to y})\mid_{x^\prime}\|=\|D(\tilde{h}^f_{x\to y})\mid_{x^\prime}-\mathrm{Id}\|<\epsilon$ for any $f\in \mathrm{Diff}^2(X)$ that is a $C^1$-$\delta$-small perturbation of $f_0$.

Therefore, we have $$|\tilde{h}^f_{x\to y}(z)-g_0(y)|=|\tilde{g}(y)-g_0(y)|$$$$\le \|D(\tilde{h}^f_{x\to y})\mid_{x^\prime}-\mathrm{Id}\|\cdot |x-z|\le \epsilon |x-\tilde{g}(x)|$$. 

\item{Case 2}: $|x-\tilde {g}(x)|\ge 1$.

In this case, for any $\epsilon_1>0$, we may take a small enough $\delta$ and consider the three ``$su$-holonomies", given as follows:

(1)$\tilde{h}^f_{x\to y}$ given by an $f$-$su$-path $\gamma_0$ between $\phi(x)$ and $\phi(y)$ such that it has $\le k$ legs and length $\le R$. This is displayed in Figure \ref{holonomy} by the blue lines.

(2) An $f_0$-holonomy $h^{f_0}_{x\to y^\prime}$ that is given by the unique $f_0$-$su$-path $\gamma_1$ which is $\epsilon_1$-near the path $\gamma_0$ such that each endpoint of $\gamma_0$ and $\gamma_1$ are in the same $\tilde \W^c_{f_0}$-leaf respectively. $\gamma_1$ is displayed in Figure \ref{holonomy} by the red lines.

(3) Another $f_0$-holonomy $h^{f_0}_{x\to y}$ that is given by an $f_0$-$su$-path $\gamma_2$ given as follows. Since $y^\prime$ is $\epsilon$-close to $y$, by Lemma \ref{ctshol} (or simply by computation in the Lie algebra), there exists an $f_0$-$su$-path $\gamma_2$ that is $\epsilon_1$-near $\gamma_1$ and connects $x$ and $y$. We pick any such path to be $\gamma_2$. In the heuristic graph this is represented by the pink lines.

\begin{figure}
 \centering
 \includegraphics[width=0.8\textwidth]{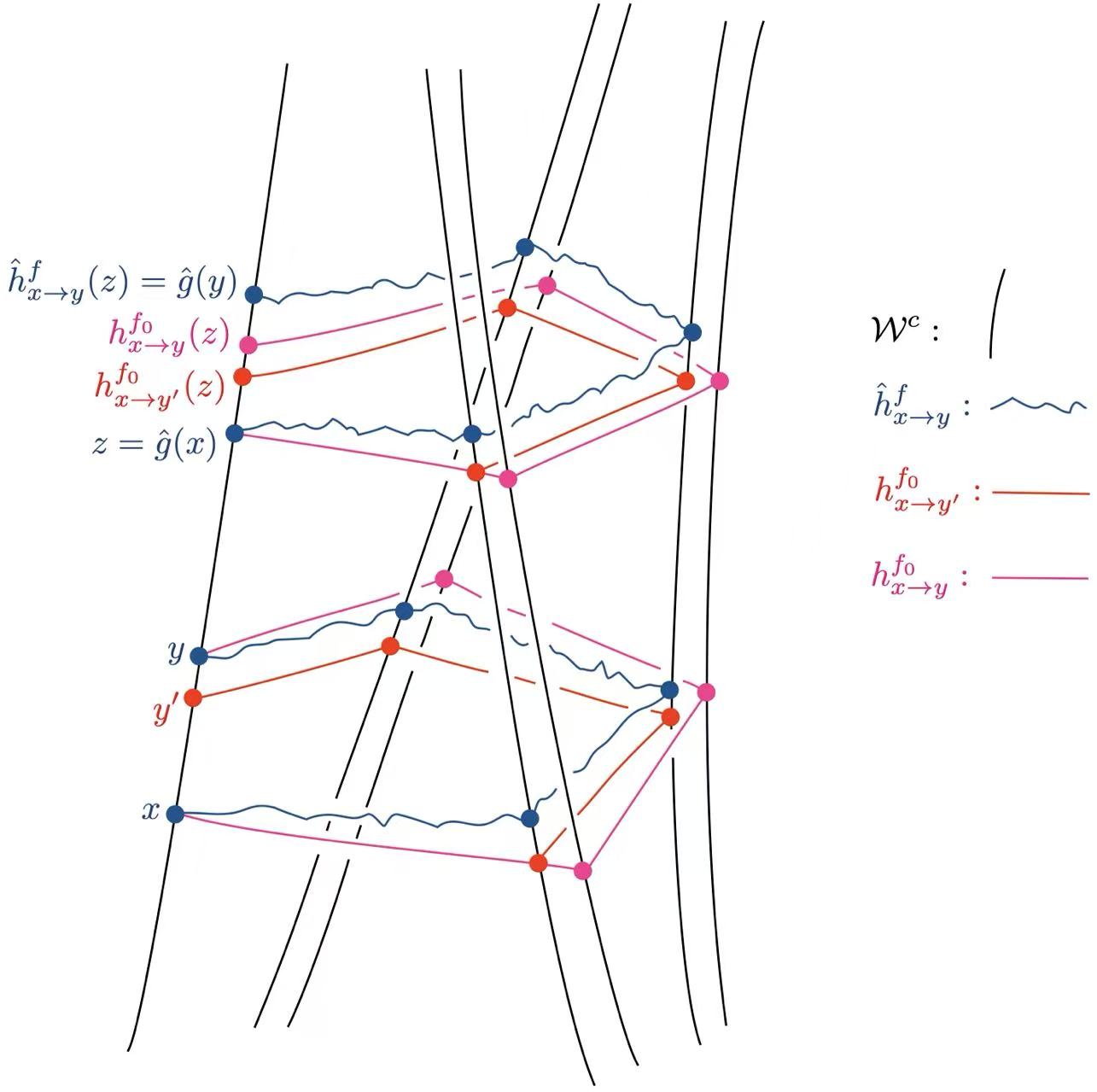}
 \caption{$su$-holonomies of $\hat{h}^f$ and $h^{f_0}$}
 \label{holonomy}
\end{figure}

Now we claim that for any $\epsilon_1>0$, there exists some $\delta>0$, such that for any $z\in\tilde\W^c_{f_0}$, we always have $$d_{\tilde\W^c_{f_0}}(\tilde{h}^f_{x\to y}(z),h^{f_0}_{x\to y^\prime}(z))\le \epsilon_1 e^{C_Gd_{\tilde \W^c_{f_0}}(x,z)}\cdot l(\gamma_1),$$ where $l(\gamma_1)$ denotes the length of $\gamma_1$. This is because the $su$-path $\gamma_1^z$ of $f_0$, given by the $su$-holonomy $h^{f_0}_{x\to y^\prime}$ and starting at $z$, has length at most $e^{C_Gd_{\tilde \W^c_{f_0}}(x,z)}\cdot l(\gamma_1)$, as given by the geometry of the leaves in the Lie group $G$, where $C_G=\max_{v\in \mathfrak h,|v|=1 } \max_{\chi_i\in \Phi}|\chi_i(v)|$ is the maximal root of a normalized vector. We can pick $\delta>0$ so that at any point, the angular distance between the stable or unstable subspace of $f_0$ and the conjugate of the stable or unstable subspace of $f$ differ by at most $\epsilon_1$ respectively.

Therefore, we have $$d_{\tilde\W^c_{f_0}}(\tilde{h}^f_{x\to y}(z),h^{f_0}_{x\to y}(z))$$$$\le d_{\tilde \W^c_{f_0}}(\tilde{h}^f_{x\to y}(z),h^{f_0}_{x\to y^\prime}(z))+d_{\tilde \W^c_{f_0}}(h^{f_0}_{x\to y}(z),h^{f_0}_{x\to y^\prime}(z))$$$$\le \epsilon_1 e^{C_Gd_{\tilde \W^c_{f_0}}(x,z)}\cdot l(\gamma_1) +d_{\W^c_{f_0}}(y,y^\prime) $$$$\le \epsilon_1 e^{C_Gd_{\tilde \W^c_{f_0}}(x,z)}\cdot l(\gamma_1) +\epsilon_1.$$

Since $l(\gamma_1)\le R$, take $\epsilon_1\le \frac{\epsilon}{2R+2}$ and take $z=\tilde{g}(x)$, we get the desired result.

\end{itemize}
\end{proof}

\section{Existence of partially hyperbolic elements in the rigid case}\label{eph}

In this section, we show that if there exists some $g\in\CZ^r_0(f)$ in a different direction from $f$, then the centralizer $\CZ^r_0(f)$ contains ``many'' topologically normally hyperbolic elements. This produces a ``genuinely higher rank'' action on $G/\Gamma$, analogous to a higher-rank restriction of the Weyl chamber flow. Such higher-rank actions are expected to be rigid, based on the results of Damjanovi\'c and Katok \cite{DK}, and Vinhage and Wang \cite{VinWang}.

\begin{defi}
 A restriction $\alpha_0$ of the Weyl chamber flow on $X=G/\Gamma$ is called \emph{genuinely higher rank} if its lifting $\tilde{\alpha}_0$ to $G$ has no homogeneous rank-$1$ factor.
\end{defi}

We say that an action $\alpha$ is a \emph{topological perturbation} of an affine action $\alpha_0$ if it preserves a family of ``Lyapunov foliations'' and acts on them dynamically like $\alpha_0$. 

\begin{defi}\label{topper}
 Let $X=G/\Gamma$ be a quotient of a genuinely higher-rank semisimple Lie group. We say that a continuous action $\alpha:\Z^2 \to \mathrm{Homeo}(X)$ is an \emph{$\epsilon,c$-topological perturbation} of a restriction $\alpha_0: \Z^2\to D$ of the Weyl chamber flow if there exists a family of continuous foliations $\W^i_\alpha$ and a center foliation $\W^c_\alpha$ that are locally transverse and preserved by the entire action $\alpha$, and constants $0<\epsilon\ll c<1$, such that:
 \begin{itemize}
 \item The foliations $\W^i_\alpha$ and $\W^c_\alpha$ are $C^0$-close to the corresponding Lyapunov foliations $\W^i_{\alpha_0}$ and center foliation $\W^c_{\alpha_0}$ of the Weyl chamber flow.

 \item For any $a\in \Z^2$, $\alpha(a)$ is projectively $\epsilon/2$-close to $\alpha_0(a)$.
 
 \item For any $a\in \Z^2$ such that $\alpha_0(a)$ is outside the $\epsilon$-cone of the Weyl chamber wall $\ker\;\chi_i$, $\alpha(a)$ contracts or expands $\W^i_\alpha$ exponentially at an exponential rate $c\cdot\chi_i(\alpha_0(a))$, i.e., for any $x\in X$, $y\in \W^i_\alpha(x)$, there exists $N>0$ (that may depend on $x,y$), such that for any $|n|>N, \,\mathrm{sign}(n)=-\mathrm{sign}(\chi_i(\alpha_0(a))),$ $$d(\alpha(na)x,\alpha(na)y)\le e^{c\chi_i(\alpha_0(a))n}.$$
 
 \item The contraction or expansion rate of $\alpha(a)$ along $\W^c_\alpha$ is $\le \epsilon^2|\alpha_0(a)|$ for any $a\in\Z^2$, i.e., for any $x\in X$, $y\in \W^c(x,loc)$ and any $n\in \Z$, $$e^{-\epsilon^2|\alpha_0(a)||n|}d(x,y)\le d(\alpha(na)x,\alpha(na)y)\le e^{\epsilon^2|\alpha_0(a)||n|}d(x,y).$$

 \item For each pair of roots $\chi_i,\chi_j$ such that $\chi_i|_{\alpha_0}\notin \R\cdot \chi_j|_{\alpha_0}$, there exist $a\in \Z^2$ inside each chamber given by $\chi_i,\chi_j$, such that $\alpha_0(a)$ is outside the $\epsilon$-cone of the Weyl chamber wall $\ker\;\chi_i$.
 \end{itemize}
\end{defi}

We now prove that $f$ is a topological perturbation of a genuinely higher-rank restriction of the Weyl chamber flow.

\begin{prop}\label{ph} 
Let $f$, $f_0$ satisfy the conditions of Proposition \ref{gph}. For any $0<\epsilon<0.01$, if $d_{C^1}(f,f_0)$ is sufficiently small, then there exists $\bar \epsilon\le \epsilon$, such that $f$ is part of an action $\alpha:\Z^2 \to \mathrm{Diff}^\infty(X)$ that is a $\bar \epsilon,c$-topological perturbation of a genuinely higher-rank restriction $\alpha_0$ of the Weyl chamber flow, where $c$ depends only on $f_0$.
\end{prop}

Later, in Section \ref{alphaph}, we will show that this action is in fact partially hyperbolic and that the coarse Lyapunov foliations agree with the fine topological foliations described above.

Heuristically, $g$ is projectively close to some translation $g_0$ in each center leaf. The translation $g_0$ exponentially contracts or expands the fine stable and unstable leaves of the Weyl chamber flow, due to properties of the Cartan subalgebra. Assuming that the geometry of the leaves of $f$ is similar to that of $f_0$, $g$ should also exponentially contract or expand the fine stable and unstable leaves of $f$, implying that $g$ is partially hyperbolic. The signature of the exponents is determined by the Weyl chamber in which $g_0$ lies.

We now try to make a rigorous argument with this picture in mind.

\subsection{Canonical choice of the topological foliations}\label{sec:choicealphafolia}
We now construct the topological foliations we shall use for $\alpha$, which we will show are uniformly expanded or contracted by many elements of $\CZ^r_0(f)$.

For any non-zero eigenvalue $\chi_i:\mathfrak{h}\to \R$ of the adjoint action $\mathrm{ad}|_{\mathfrak{h}}$ on the Lie algebra $\mathfrak{g}$, let $\W^i_{f_0}$ denote the foliation given by the unique integration of the eigenspace $\mathfrak{g}_i$, and let $\W^{i,c}_{f_0}$ denote the foliation given by the unique integration of $\mathfrak{g}_i \oplus \mathfrak{h}$ in $G$. 

We define 
\[
\W^{i,c}_\# = \phi(\W^{i,c}_{f_0})
\] 
as the topological leaves given by the leaf conjugacy $\phi$, and set 
\[
\W^i_\# := \W^{i,c}_\# \cap \W^{s/u}_f.
\]
We claim that $\W^i_\#$ is a topological sub-foliation of $\W^{s/u}_f$. We assume without loss of generality that $\chi_i(f_0)<0$. Then in a local piece of leaf $\W^{c,s}_f(x,loc)$, the leaves of $\W^i_\#$ are exactly  $\W^i_\#(x,loc)=\{\pi_s(\phi(y)): y\in \W^i_{f_0}(\phi^{-1}(x),loc)\}$, where $\pi_s$ is the local projection $\W^{c,s}_f(x,loc)\to \W^{s}_f(x,loc)$, given by the holonomy along center foliations. Moreover, since $\phi$ is the leaf conjugacy and $\W^s_f$ and $\W^c_f$ form a local product structure, the map $\pi_s\circ \phi$ is injective. This shows that $\W^i_\#$ a topological sub-foliation of $\W^{s/u}_f$ whose leaves are $C^0$-close and locally homeomorphic to leaves of $\W^i_{f_0}$.

For any $g\in \CZ^r_0(f)$, $g$ preserves the topological foliation $\W^i_\#$ because $g$ fixes the center leaves, hence fixing $\W^{i,c}_\#$, and also preserves the foliations $\W^s_f$ and $\W^u_f$.

The following sections will establish exponential contraction and expansion along these topological foliations.

\subsection{Exponential contraction on leaves}

We now prove that $g\in\CZ^r_0(f)$ acts by contraction and expansion as expected on the topological leaves $\W^i_\#$ in the case that some $g$ is projectively close to a translation.

\begin{lem}\label{gcontr} 
For any $\epsilon>0$, let $g\in \CZ^r_0(f)$ be projectively $\epsilon/2$-close to a translation $g_0$ not in the $\epsilon$-cone of the Weyl chamber wall $\ker(\chi_i)$. Then for every $x\in X$ and $y\in \W^i_\#(x)$, there exists $N>0$ such that for any $n>N$:
\[
d(g^n(x),g^n(y)) \le e^{\chi_i(g_0) \eta \cdot n/4}, \quad \text{if } \chi_i(g_0)<0,
\] 
\[
d(g^{-n}(x),g^{-n}(y)) \le e^{-\chi_i(g_0) \eta \cdot n/4}, \quad \text{if } \chi_i(g_0)>0.
\] 
Here $\eta<1$ is the bi-H\"older constant of the leaf conjugacy $\phi$.
\end{lem}

\begin{proof}[Proof of Lemma \ref{gcontr}]
See Figure \ref{lypest} for a more intuitive explanation.

\begin{figure}[h]
 \centering
 \includegraphics[width=0.8\textwidth]{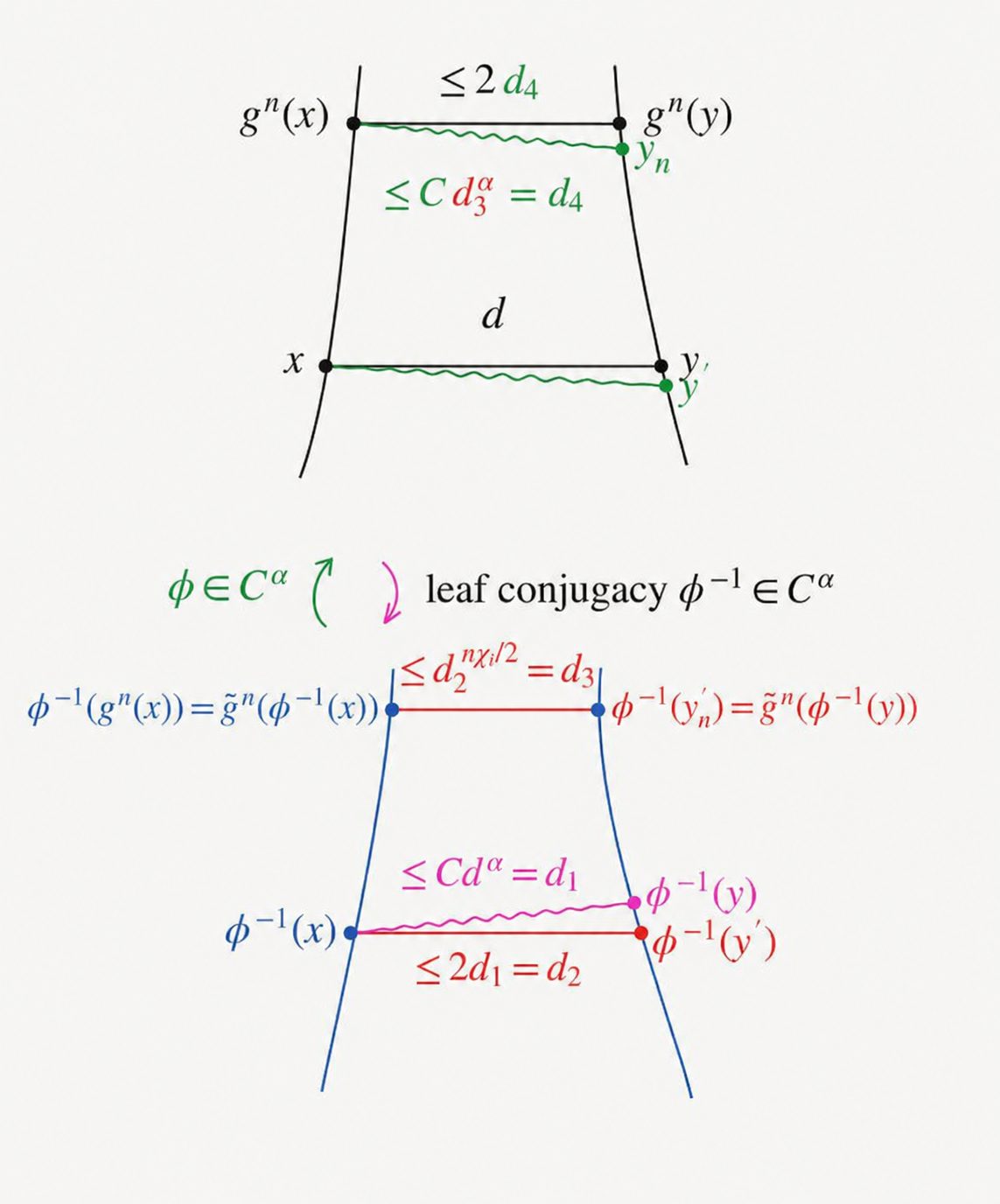}
 \caption{Contraction rate of $g$ along $\W^i_\#$}
 \label{lypest}
\end{figure}

We restrict our discussion to the case $\chi_i(g_0)<0$; the other case follows by considering $g^{-1}$ instead of $g$.

Denote by $\rho(g_0):=\log(g_0(x)x^{-1})$ the translation vector of $g_0$ in $\W^c_f\simeq \R^k$. Denote by $\tilde{g}$ the lifting of $\phi^{-1} g\phi$ where $\phi$ is the leaf conjugacy.

 For any $n\in \N$, by Corollary \ref{trorf} and Lemma \ref{transadd}, we have 
\begin{equation}\label{neartrans}
|\tilde{g}^n(z) - z - n \rho(g_0)| \le \epsilon n \|\rho(g_0)\|/2
\end{equation}
for $z \in X$. Since the $g_0$ is not in the $\epsilon$-cone of the Weyl chamber wall $\ker(\chi_i)$, we have $\chi_i(t) \le \chi_i(g_0)/2$ for any $t \in \mathfrak{h}$ with $d(t,\rho(g_0))<\epsilon\|\rho(g_0)\|/2$. By the geometry of the $f_0$ foliations, for any $z \in X$ and any $w \in \W^i_{f_0}(z)$, we have
\[
d(\hat{g}^n(z), \W^c_{f_0}(w)\cap \W^i_{f_0}(\hat{g}^n(z))) \le e^{n \chi_i(g_0)/2} d(z,w).
\]

Using the H\"older continuity of the leaf conjugacy $\phi$, any $x \in X$ and every $y \in \W^i_\#$ with $d(x,y)\le r$, set 
\[
y' = \phi(\W^i_{f_0}(\phi^{-1}x)) \cap \W^c_f(y), \quad 
y'_n = \phi(\W^i_{f_0}(\phi^{-1}g^n(x))) \cap \W^c_f(g^n(y)).
\] 

Suppose $\W^i_{\#}$ is a subfoliation of $\W^*_f,*=s,u$. We note that the angle between $E^c_f,E^*_f$ are close $\pi/2$, so there exists $r>0$, such that if $d_{\W^{c*}_f}(z,w)<r$, then we have $d(z,\W^c(w,loc))\ge 2 d_{\W^*_f}(x,\W^*_f(z, loc)\cap\W^c(w,loc)) $. Similar equations hold for $f_0$ for the same reason. 

This shows that, for any $y\in \W^i_\#(x)$, $d_{\W^{c*}_f}(x,y)\le r$, we have
\[
d(g^n(x), g^n(y)) \le 2 d(g^n(x), y'_n) \le 2 C d(\hat{g}^n(\phi^{-1}x), \phi^{-1}(y'_n))^\eta
\]
\[
\le 2 C (e^{n \chi_i(g_0) /2} d(\phi^{-1}(x), \phi^{-1} y'))^\eta
\le 4 C e^{n \eta \chi_i(g_0) /2} d(\phi^{-1}(x), \phi^{-1}(y))^\eta\]\[
\le 4 C^2 e^{n \eta \chi_i(g_0) /2} d(x,y)^{\eta^2}.
\]

When $d(x,y)$ is larger, by taking intermediate pieces of length $r$, and choosing $n>N$, where $N$ is sufficiently large, we obtain the desired estimate.
\end{proof}

\subsection{Small exponent in the center }
In this section, we show that elements $g\in \CZ^r_0(f)$ have very small exponential expansion or contraction along the center foliation of $f$.

\begin{prop}\label{prop:smallexp}
 Given constants $\epsilon>0$, $R>1$ and a base point $x_0\in X$. There exists $\delta>0$, such that if $d_{C^1}(f,f_0)<\delta$, then for any $g\in \CZ^r_0(f)$, such that $d_{\tilde\W^c_f}(\tilde g(\tilde x_0),\tilde x_0)<R$, we have $$e^{-\epsilon |n|}\|v\|\le \|Dg^n(v)\|\le e^{\epsilon |n|}\|v\|$$ for any $n\in \Z$, any $x\in X$ and any $v\in E^c_f(x)$.
\end{prop}

\begin{proof}
Fix a base point $x_0\in X$. We note that each leaf of $\tilde \W^c_f$ is isomorphic to $\R^k$, where $k=\mathrm{rank} (G)$, so we may treat $Dg|_{E^c_f}(y)$ just as matrices for $y\in X$. We denote by $A=Dg|_{E^c_f}(x_0)$. 

 For any $y\in X$, we pick a $f$-$su$-path $\gamma$ from $x_0$ to $y$ of uniformly bounded legs and length. Then the holonomy $h^f_\gamma: \tilde \W^c_f(x_0) \to \tilde \W^c_f(y)$ satisfies $h^f_\gamma(g(z))=g(h^f_\gamma(z))$ for any $z\in \tilde \W^c_f(x_0)$. Since $d_{\tilde\W^c_f}(\tilde g(\tilde x_0),\tilde x_0)<R$, the holonomy path from $\tilde g(\tilde x_0)$ to $\tilde g(\tilde y)$ given by the holonomy $h^f_\gamma$ is also of uniformly bounded legs and length depending only on $G$ and $R$.
 
 Therefore by Theorem \ref{c1hol}, and because the differentials of holonomies of $f_0$ are the identity, for any $\epsilon_1>0$, there exists $\delta>0$ such that we have $$\|(Dh^f_\gamma(\tilde x_0))^{-1}-\mathrm{Id}\|,\|Dh^f_\gamma(\tilde g(\tilde x_0))-\mathrm{Id}\|\le \epsilon_1.$$

 Again using $h^f_\gamma(g(z))=g(h^f_\gamma(z))$ and plugging in $z=x_0$, we get $$Dg|_{E^c_f}(y)=Dh^f_\gamma(\tilde{g}(\tilde x_0))\cdot Dg|_{E^c_f}(x_0)\cdot (Dh^f_\gamma(\tilde x_0))^{-1}.$$

 Therefore, we show that the differential of $g$ restricted to $E^c_f$ is almost constant, i.e., we have 
 \begin{multline}\label{equ:Dg bound}
 \|Dg|_{E^c_f}(y)-A\|= \|Dh^f_\gamma(\tilde{g}(\tilde x_0))\cdot A\cdot (Dh^f_\gamma(\tilde x_0))^{-1}-A\| \\
 \le\|(Dh^f_\gamma(\tilde{g}(\tilde x_0))-\mathrm{Id})\cdot A\cdot (Dh^f_\gamma(\tilde x_0))^{-1}\|+\|A\cdot ((Dh^f_\gamma(\tilde x_0))^{-1}-\mathrm{Id})\|\\\le 3\epsilon_1\|A\|.
 \end{multline}

 Moreover, for any $\epsilon_0>0$, if $d_{C^1}(f,f_0)$ is sufficiently small, then by Proposition \ref{z2:dtof0}, we have $d(\tilde g(y),g_0(y))\le \epsilon_0 \,d(g_0(x_0),x_0)$, for any $y\in X$. 
 
 Combining the bound on the differential and the bound on the map, we now claim $\|A\|\le \frac{1}{1-3\epsilon_1}$. Indeed, for any $x\in G$, identifying $\tilde \W^c_f(x)\simeq \R^k$, then for any $v\in \R^k,\|v\|=1$, let $y=x+tv$, then we have $$\|g(y)-g(x)-tAv\|=\|\int_{s=0}^t(Dg|_{E^c_f(x+sv)}-A)(v)\cdot ds\|\le 3\epsilon_1 \|A\|t.$$ On the other hand, $$\|g(y)-g(x)-tAv\|\ge \|g_0(y)-g_0(x)-tAv\|-\|g(y)-g_0(y)\|-\|g(x)-g_0(x)\|$$ $$=\|y-x-tAv\|-\|g(y)-g_0(y)\|-\|g(x)-g_0(x)\|$$$$=t\|Av-v\|-\|g(y)-g_0(y)\|-\|g(x)-g_0(x)\|$$$$\ge t\|(A-\mathrm{Id})v\|-2\epsilon_0R$$ Combining these and let $t\to \infty$, we get $\|A-\mathrm{Id}\|\le 3\epsilon_1 \|A\|$, which implies that $\|A\|\le \frac{1}{1-3\epsilon_1}$. Therefore, by Equation \ref{equ:Dg bound} we have $\|Dg|_{E^c_f}(y)\|\le \frac{1+3\epsilon_1}{1-3\epsilon_1}$ for any $y\in X$. Take $\epsilon_1$ small enough so that $\frac{1+3\epsilon_1}{1-3\epsilon_1}<e^\epsilon$, we get the upper bound on the growth.

 The same argument applied to $g^{-1}$ instead of $g$ gives the lower bound on the growth.
\end{proof}

\subsection{Weyl chamber picture and choice of $\alpha_0$}\label{sec:weylchamchoice}

For any $\Z^2$-restriction $\alpha_1$ of the Weyl chamber flow, we consider the 2-planes $P$ in $\mathfrak h$ spanned by the two generators, and the restriction of the Weyl chamber walls $\ker \,\chi_i,1\le i\le K$ in $P$. If $P$ does not lie in any of the Weyl chamber walls, then the restrictions
$\ker\chi_i\cap P$ give a finite collection of not necessarily distinct lines
$\ell_1,\ldots,\ell_K\subset P$.
We use the canonical cyclic order  on the projective line of $P$ (where the order among overlapping lines are chosen arbitrarily). For each ordered pair $(i,j)$, let $\theta_{ij}\in[0,\pi)$ denote the oriented angle
from $\ell_i$ to $\ell_j$ in this cyclic order, with $\theta_{ii}=0$. For any $\Z^2$ restriction $\alpha_1$, we denote by $\theta(P)=\{\theta_{ij}\}$ the angles between the lines $\ell_i$ in $P$.

Consider the set $\mathcal P$ which is the set of $2-$planes in $\mathfrak{h}$ that contain $f_0$. For any pair of walls $i,j$, the angle $\theta_{ij}:\mathcal P\to \R_{\ge 0}$ is a piecewise-smooth function that has no local minimum except zero as $P$ ranges in $\mathcal P$. 

Let $Z_{ij}=\{P\subset \mathcal{P}:\theta_{ij}(P)=0\}$ be the zero set of $\theta_{ij}$. Then $Z_{ij}$ is simply a hyperplane in $\mathcal{P}\simeq P(\R^{k-1})$. Furthermore, there exists a constant $C_1>1$ such that $\frac{1}{C_1}\cdot d(P,Z_{ij})\le \theta_{ij}(P)\le C_1\cdot d(P,Z_{ij})$.

For each subset $I\subset \{(i,j): 1\le i<j\le K\}$, we denote by $Z_I=\cap_{(i,j)\in I}Z_{ij}$. Then each $Z_I$ is also a projected subspace of $\mathcal{P}\simeq P(\R^{k-1})$, or the empty set. Let $t=\max\{|I|:Z_I\neq \varnothing\}$.

Fix $0<\epsilon<1$ as in Proposition \ref{ph}, then we have the following.

\begin{lem}\label{lem:anglecontrol}
 There exists $\epsilon/C_1>\epsilon_t>\epsilon_{t-1}>\cdots>\epsilon_1$, such that if $d(P,Z_I)<\epsilon_{|I|}$ and either $d(P,Z_{J})\ge\epsilon_{|J|}$ for any $J\supsetneq I$, or $|I|=t$; then there exists $P^\prime \in Z_{I}$, such that $d(P,P^\prime)=d(P,Z_I)$, while $d(P^\prime,Z_{ij})\ge 8C_1\epsilon_{|I|}$, for any $(i,j)\notin I$. 
\end{lem}
\begin{proof}
 We shall pick $\epsilon_i$ one by one, starting with $\epsilon_t$.
 
 We denote by $B(Y,r)$ the open $r$-neighborhood of $Y$ in $ \mathcal P$.
 
 We first pick $\epsilon_t$. The sets $Z_{I},|I|=t$ is compact and $Z_{I}\cap Z_{ij}=\varnothing$ if $(i,j)\notin I$. Therefore, for each $Z_{I}$, \begin{equation}\label{equ:emint}
 B(Z_{I},r)\cap (\cup_{(i,j)\notin I} Z_{ij})=\varnothing
 \end{equation} when $r$ is sufficiently small. Since there are only finitely many $Z_{I}\neq\varnothing,|I|=t$, we may pick $\epsilon_t<\epsilon/C_1$ small enough, so that Equation \ref{equ:emint} holds for $r=8C_1\epsilon_t$ for any $|I|=t$. 

 If we already picked $\epsilon_t,\cdots,\epsilon_{s+1}$, we take $\epsilon_s$ as follows. For each $|I|=s$, the set $Z_{I}^\prime=Z_{I}-\cup_{J\supsetneq I} B(Z_{J},\epsilon_{|J|})$ is compact, and disjoint from any $Z_{ij}, (i,j)\notin I$. Therefore, \begin{equation}\label{equ:emint2}
 B(Z_{I}^\prime,r)\cap (\cup_{(i,j)\notin I} Z_{ij})=\varnothing
 \end{equation} when $r$ is sufficiently small. Again, since there are only finitely many such sets, we may pick $\epsilon_s<\epsilon_{s+1}$ small enough, so that Equation \ref{equ:emint2} holds for $r=8C_1\epsilon_s$.

 Picking $\epsilon_t$ to $\epsilon_1$ like this, we have for any $P\in \mathcal P$, satisfying the assumption, then $P\in B(Z^\prime(I),\epsilon_{|I|})$, so there exists $P^\prime\in Z^\prime_{I}\subset Z_{I}$, such that $d(P^\prime, \cup_{(i,j)\notin I} Z_{ij})\ge 8C_1\epsilon_{|I|}$ which proves our claim.
\end{proof}

We note that all the constants defined above depend only on $G$, $f_0$ and $\epsilon$. We choose $\delta$ small enough, so that any $g\in \CZ^r_0(f)$ is projectively $\frac{\epsilon_1}{8C_1}$-close to the left translation by $(\tilde g(x)x^{-1})$ if $d_{\tilde \W^c_f}(\tilde g(x),x)\le 3\rho(f_0)$. (by Proposition \ref{z2:dtof0}).

Now we suppose $g\in \CZ^r_0(f)$ is projectively $0.01$-close to a translation outside the $0.1$-cone of $f_0$ in every simple component of the Lie algebra. Let $\alpha=\langle f,g\rangle$. We take $g_0$ to be the left translation of $(\tilde g(x)x^{-1})$. 

Let $\mathcal P$ be the plane generated by $\langle f_0,g_0\rangle$. 

If $\theta_{ij}\ge \epsilon_1/C_1$ for any $i\neq j$, then we take $\alpha_0=\langle f_0,g_0\rangle$, and $\bar{\epsilon}=\frac{\epsilon_1}{4C_1}$. 

Then $g$ is projectively $ \bar\epsilon/2$-close to $g_0$ and $f$ is arbitrarily close to $f_0$, by Lemma \ref{transadd} and our assumption $g$ is not in the $0.1$-cone of $f_0$, we have for any $a\in \Z^2$, $\alpha(a)$ is projectively $\bar\epsilon/2$-close to $\alpha_0(a)$.

Moreover, by our choice of $\bar{\epsilon}$, each pair of Weyl chamber walls of $\alpha_0$ has angle at least $ \epsilon_1/C_1=4\bar \epsilon$. This shows that in each Weyl chamber of $\alpha_0$, there exists $a\in \Z^2$ outside the $\bar\epsilon$-cone of the walls. 

If $\theta_{ij}< \epsilon_1/C_1$ for some pair $(i,j)$, then we have $d(P,Z_{ij})< \epsilon_1$. We pick $|I|$ to be a maximal subset of $\{(i,j):1\le i,j\le K\}$, such that $d(P,Z_I)< \epsilon_{|I|}$ (the set of such $I$ is non-empty by our assumption). Then by Lemma \ref{lem:anglecontrol}, there exists $P^\prime \in Z_{I}$, such that $d(P,P^\prime)=d(P,Z_{I})$ and $d(P^\prime, Z_{ij})\ge 8C_1\epsilon_{|I|}$. Take $\bar \epsilon=2\epsilon_{|I|}$. 

Pick $\alpha_0=\langle f_0,g_0^\prime\rangle$ to be the action generated by $f_0$ and $g_0^\prime$, where $g_0^\prime\in P^\prime$ realizes the minimal distance to $g_0$ in $P^\prime$. By our assumption, $g$ is projectively $ \epsilon_{|I|}/4$-close to $g_0$, which implies that $g$ is projectively $ \frac{\bar\epsilon}{20}$-close to $g_0^\prime$ (since $g_0$ is not in the $0.1$-cone of $f_0$). Since $f$ is also in an arbitrarily small cone of $f_0$ when $\delta$ is sufficiently small, by Lemma \ref{transadd} and our assumption $g$ is not in the $0.1$-cone of $f_0$, we have for any $a\in \Z^2$, $\alpha(a)$ is projectively $\bar\epsilon/2$-close to $\alpha_0(a)$. 

Moreover, in the plane $P^\prime$, by Lemma \ref{lem:anglecontrol}, each non-zero angle $\theta_{ij}(P^\prime)\ge\frac{d(P^\prime,Z_{ij})}{C_1}\ge 8\epsilon_{|I|}=4\bar\epsilon$, therefore, for each Weyl chamber of $\alpha_0$, there exists $a$ such that $\alpha_0(a)$ lies in that chamber and $\alpha_0(a)$ is not in the $\bar\epsilon$-cone of the walls of that chamber.

\subsection{Proof of Proposition \ref{ph}}

 Now we give a proof of Proposition \ref{ph}.

\begin{proof}[Proof of Proposition \ref{ph}]

 Take $\alpha=\langle f,g\rangle $ to be the $\Z^2$ action generated by $f$ and $g$, and take $\W^i_\alpha = \W^i_\#$ as constructed in section \ref{sec:choicealphafolia}. By construction, $\W^i_\#$ is $C^0$-close to $\W^i_{\alpha_0}$ because the leaf conjugacy is $C^0$-close to the identity and $\W^{s/u}_f$ is $C^0$-close to $\W^{s/u}_{f_0}$.

 We construct $\alpha_0$ and pick $\bar\epsilon<\epsilon$ as in Section \ref{sec:weylchamchoice}. We proved in Section \ref{sec:weylchamchoice} that for any $a\in \Z^2$, $\alpha(a)$ is projectively $\bar\epsilon/2$-close to $\alpha_0(a)$.

 Pick $c=\eta/4$, where $\eta$ is the bi-H\"older constant, which can be chosen to depend only on $f_0$. Then Lemma \ref{gcontr} provides exponential contraction and expansion along the Lyapunov foliations at rate $c\cdot \chi_i(\alpha_0(a))$, if $\alpha_0(a)$ is not in the $\epsilon$-cone of the Weyl chamber wall $\ker\chi_i$. 

 Moreover, we have $d_{\tilde \W^c_f}(\tilde g(\tilde x_0),\tilde x_0)<3\rho(f_0)$. Therefore, Proposition \ref{prop:smallexp} implies that when $\delta$ is sufficiently small, we always have $e^{-\frac{\bar\epsilon^2 \|\rho(f_0)\|}{100}n}\le \|Dg^n|_{E^c_f}v\|\le e^{\frac{\bar\epsilon^2 \|\rho(f_0)\|}{100}n}$ for any $\|v\|=1$, independent of the choice of $g$. We have the same result for $f$. Therefore, for any element $a\in \Z^2$, $\alpha(a)=f^kg^\ell$, we have $e^{-\frac{\bar\epsilon^2 \|\rho(f_0)\|}{100}(|k|+|l|)n}\le \|D\alpha(a)^n|_{E^c_f}v\|\le e^{\frac{\bar\epsilon^2 \|\rho(f_0)\|}{100}(|k|+|l|)n}$. Since $g$ is not in the $0.1$-cone of $f_0$, and $\|\rho(f_0)\|\le \|\rho(\alpha_0(0,1))\|\le 2\|\rho(f_0)\|$, we get from here that the exponent of $\alpha(a)$ along $E^c_f$ is smaller than $\bar\epsilon^2|\alpha_0(a)|$.

 In Section \ref{sec:weylchamchoice}, we established that there are elements not in the $\bar\epsilon$-cone of the walls in every Weyl chamber cone of $\alpha_0$. This proves that $\alpha$ is a $(\bar\epsilon,c)$-topological perturbation of $\alpha_0$.

 Finally, note that we picked $g$ to be not in the $0.1$-cone of $f_0$ in each simple component of $G$, and $g$ is $\bar\epsilon$-close to $\alpha_0(0,1)$, where $\bar\epsilon<0.01$. This shows that $\alpha_0(0,1)$ is not in the $0.05$-cone of $f_0=\alpha_0(1,0)$ in each simple component of $G$, and so $\alpha_0$ is genuinely higher rank.
\end{proof}

\section{Rigidity of Weyl chamber flow}\label{wcf}

In this section, we conclude the proof of Proposition \ref{gph} using results on rigidity of higher-rank restrictions of the Weyl chamber flow, as established in \cite{DK}, \cite{Vinhage}, \cite{VinWang}, etc.

Recall that a semisimple Lie group $G$ is \emph{genuinely higher rank} if each of its simple components has rank $\ge 2$. Let $V \subset \mathfrak{h}$ be a subgroup of the Cartan subalgebra inducing an action
\[
\alpha_0: V \to \mathrm{Diff}(G/\Gamma), \quad v \mapsto L_{\exp(v)},
\]
which is a restriction of the Weyl chamber flow. We call this restriction \emph{genuinely higher rank} if its lifting to $G$ has no rank 1 factor. It is proved in Theorem 3.1 of \cite{VinWang} that a $C^1$-small smooth perturbation of a genuinely higher-rank restriction of the Weyl chamber flow $\alpha_0$ on $X = G/\Gamma$ is $C^\infty$ conjugate to a restriction of a Weyl chamber flow, provided $\alpha_0$ includes a generic element.

In our setting, we take a $\Z^2$-action generated by $f$ and $g \in \CZ^r_0(f)$ not in the same direction as $f$, as in Section \ref{eph}. As noted in Proposition \ref{ph}, $\alpha$ is not a priori a $C^1$-small perturbation of a restriction $\alpha_0$ of the Weyl chamber flow, so we cannot directly apply Theorem 3.1 of \cite{VinWang}. However, by Proposition \ref{ph}, $\alpha$ is a topological perturbation of a genuinely higher-rank action $\alpha_0$ in the sense of Definition \ref{topper}.

Therefore, the methods of \cite{VinWang} can still be applied. Below we present the outline, highlighting only the arguments that require adjustment.

\subsection{Cocycle formulation}

Given a group action $\alpha_1: A \to \mathrm{Diff}(X)$, a cocycle $\beta: A \times X \to D$ is a map satisfying
\[
\beta(ab, x) = \beta(a, \alpha_1(b)(x)) \, \beta(b, x), \quad \forall a, b \in A,
\]
taking values in the diagonal group $D$. A cocycle is cohomologous to a constant if there exists a homomorphism $s: A \to D$ and a transfer map $H: X \to D$ such that
\[
\beta(a, x) = H(\alpha_1(a)(x)) \, s(a) \, H(x)^{-1}, \quad \forall a \in A, x \in X.
\]

We consider the following cocycle:

\begin{prop}
Let $\alpha$ be the $\Z^2$-action as in Proposition \ref{ph}, and let $\hat{\alpha}$ denote the leaf conjugacy of $\alpha$:
\begin{equation}\label{hatalpha}
\hat{\alpha}(a) = \phi^{-1} \alpha(a) \phi,
\end{equation}
where $\phi$ is the leaf conjugacy from $f_0$ to $f$. Then the map
\begin{equation}\label{cocbeta}
\hat{\beta}(a, x) = \widetilde{\hat{\alpha}(a)}(\tilde x) \, \tilde x^{-1} \, \alpha_0(a)^{-1}
\end{equation}
is a H\"older cocycle over $\hat{\alpha}$ taking values in the Cartan subgroup $D$. Here, $\alpha_0(a)$ denotes the unique element in $D$ such that $\alpha_0(a)(x) = \alpha_0(a) \cdot x$, $\widetilde{\hat{\alpha}(a)}$ is the lift of the homeomorphism $\hat{\alpha}(a)$ to $G$ that fixes the center leaves of $f_0$, and $\tilde x$ is a lifting of $x$.
\end{prop}

\begin{proof}
The translation $\widetilde{\hat{\alpha}(a)}(\tilde x) \, \tilde x^{-1}\in D$ is independent of the lifting $\tilde x$, since our choice of $\alpha$ consists of elements in $\CZ^r_0(f)$.

We only need to check that $\hat{\beta}$ satisfies the cocycle identity. Using the definition and the fact that $D$ is abelian:
\[
\begin{aligned}
\hat{\beta}(ab, x) &= \widetilde{\hat{\alpha}(ab)}(\tilde x) \, \tilde x^{-1} \, \alpha_0(ab)^{-1} \\
&= \widetilde{\hat{\alpha}(a)}(\widetilde{\hat{\alpha}(b)}(\tilde x)) \, \tilde x^{-1} \, \alpha_0(b)^{-1} \alpha_0(a)^{-1} \\
&= \widetilde{\hat{\alpha}(a)}(\widetilde{\hat{\alpha}(b)}(\tilde x)) \, (\widetilde{\hat{\alpha}(b)}(\tilde x))^{-1} \alpha_0(a)^{-1} \cdot (\widetilde{\hat{\alpha}(b)}(\tilde x) \, \tilde x^{-1}) \, \alpha_0(b)^{-1} \\
&= \hat{\beta}(a, \hat{\alpha}(b)(x)) \cdot \hat{\beta}(b,  x).
\end{aligned}
\]
This proves the claim.
\end{proof}

We now state the cocycle rigidity result:

\begin{thm}\label{cocrig}
Let $\hat{\alpha}$ be the $\Z^2$-action given in Equation \eqref{hatalpha}. Then any H\"older continuous cocycle over $\hat{\alpha}$
\[
\beta: \Z^2 \times X \to D
\]
that is $C^0$-close to a constant is cohomologous to a constant via a H\"older continuous transfer map $H$.
\end{thm}

Proposition \ref{gph} is a direct corollary of Theorem \ref{cocrig} applied to the cocycle $\hat{\beta}$ in Equation \eqref{cocbeta}. The proof follows similarly to Theorem 1.1 in \cite{DK}.

\begin{proof}[Proof of Proposition \ref{gph}]
Take $\hat{\beta}$ as in Equation \eqref{cocbeta}. By Proposition \ref{ph} and Theorem \ref{cocrig}, $\hat{\beta}$ is cohomologous to a constant via a H\"older continuous transfer map $H$. Define
\[
h_1: X \to X, \quad h_1(x) = H(x)^{-1} \cdot x.
\]
Then for any $\hat{g} = \hat{\alpha}(a)$, $a \in \Z^2$,
\begin{equation}\label{chomg}
h_1 \circ \hat{g}(x) = s(a) \, \alpha_0(a) \cdot h_1(x), \quad \forall x \in X.
\end{equation}

We claim that $h_1$ is a homeomorphism.
Since we chose $\alpha(a)\in \CZ^r_0(f)$ for any $ a\in \Z^2$, we may lift $h_1$ to the cover $G$. We note that if $h_1(x)=h_1(y)$, then $h_1(h^s(x))=h_1(h^s(y))$ for any $s$-holonomy $h^s$ for some $\hat{g}\in A$. This is because the semiconjugacy maps fixes the center foliations and sends topological stable foliations to stable foliations. Therefore $h_1(h^{(n)}_{x\to y}(x))=h_1(x)$ for any $su$-path of $\hat{f}$ from $x$ to $y$ and any iterate $n$, that is, the whole $h_{x\to y}$ orbit of $x$ is in the preimage $h_1^{-1}(\{h_1(x)\})$ which is bounded. However, since the holonomy group acts freely on a generic center leaf of $G$ by Proposition \ref{hologroup}, that is impossible unless $x=y$. Moreover, by our construction, $h_1$ is bounded from identity when lifted to $G$, since the coboundary $H(x)$ is uniformly bounded. This shows that $h_1$ must also be surjective.

 Then, since $\hat{g}$ is H\"older conjugate to an element of the Weyl chamber flow on $X$ for any $g=\alpha(a)$, in particular, $\hat{f}$, and hence $f$ is also H\"older conjugate to an element of the Weyl chamber flow. We shall show in Theorem \ref{parthyp} that $\alpha$ is partially hyperbolic. 
 
 Therefore, the leaves of coarse Lyapunov foliations of $\alpha$ are $C^\infty$ smooth. Then by standard normal form theory on algebraic manifolds, we may upgrade the conjugacy $h_1$ to being smooth. The smoothness of $h_1$ along the coarse Lyapunov foliations follows as in Step 4 of Section 2.2.3 in \cite{KS}. The global smoothness follows from Theorem 2.1 of \cite{KSpat}. A proof is provided in Appendix \ref{hsmooth}.
\end{proof}

\subsection{Lyapunov cycles in general}

We begin with some analysis of the group of Lyapunov cycles for an action $\alpha$ with topological Lyapunov foliations.

A path in $X$ whose legs lie entirely within topological Lyapunov foliations of $\alpha$ is called a \emph{Lyapunov path} of $\alpha$. We denote such a path by its endpoints $[x_0, x_1, \dots, x_k]$, where $x_i \in \W^j_\alpha(x_{i-1})$ for some $j$ and all $1 \le i \le k$. Here, 
\[
\W^j_\alpha := \bigcap_{a: \chi_j(a) < 0} \W^s_{\alpha(a)}
\] 
are the Lyapunov foliations of $\alpha$.

A Lyapunov path of $\alpha$ is called an \emph{$\alpha$-cycle} if it starts and ends at the same point. We denote by $C_{x_0}(\alpha)$ the set of $\alpha$-cycles based at $x_0$, and by $C^G_{x_0}(\alpha)$ the set of $\alpha$-cycles whose lifting to $G$ remains closed.

A $su$-cycle in $C_{x_0}(\alpha)$ is said to be \emph{stable} if there exists $g \in \mathrm{Im}(\alpha)$ such that the entire path lies in the stable manifold $\W^s_g(x_0)$, generated by $\W^i_\alpha$ with $\chi_i(g) < 0$. We denote by $S_{x_0}(\alpha)$ the closure of the normal subgroup of $C_{x_0}(\alpha)$ generated by stable cycles. Note that stable cycles are contractible (their image under sufficiently large iteration of $g$ lies in a small ball), so we always have
\[
S_{x_0}(\alpha) \subset C^G_{x_0}(\alpha).
\]

\begin{defi}[$su$-cycle functional]
Given a cocycle $\beta$ of a group action $\alpha$ on $X$ with topological Lyapunov foliations, define the $su$-cycle functional
\[
P_\beta: C_{x_0}(\alpha) \to D
\]
by
\[
P_\beta: [x_0, x_1, \dots, x_k = x_0] \mapsto \prod_{i=0}^{k-1} p_\beta(x_i, x_{i+1}),
\]
where 
\[
p_\beta(x, y) = \lim_{n \to \infty} \beta(na, x)^{-1} \beta(na, y)
\]
for $y \in \W^i_\alpha(x)$ and $a \in A$ such that $\chi_i(\alpha(a)) < 0$.
\end{defi}

By H\"older continuity of $\beta$ and exponential contraction for $y\in \W^i_\alpha$, it can be shown that $p_\beta(x,y)$ is well-defined and does not depend on the choice of a contracting $a\in A$. (See Proposition 2 of \cite{DK05}).

The following proposition shows that a cocycle is cohomologous to a constant if and only if its $su$-cycle functional is trivial. The proof is identical to Proposition 4 and Proposition 10 in \cite{DK05}, even though the original statement concerns partially hyperbolic actions rather than actions with ``topological Lyapunov foliations".

\begin{prop}\label{consttrivialfunctional}
Let $\alpha: A \to \mathrm{Diff}(X)$ be an action with accessible topological Lyapunov foliations. Then a H\"older cocycle $\beta: A \times X \to D$ is cohomologous to a constant via a H\"older continuous transfer map $H$ if and only if $P_\beta: C_{x_0}(\alpha) \to D$ is trivial for some $x_0 \in X$.
\end{prop}

By definition, any stable cycle lies in the kernel of the $su$-cycle functional:

\begin{lem}
We have 
\[
S_{x_0}(\alpha) \subset \ker P_\beta
\]
for any H\"older cocycle $\beta$ of $\alpha$.
\end{lem}

\begin{proof}
For a stable path $\gamma = [x_0, x_1, \dots, x_k = x_0] \in S_{x_0}(\alpha)$, pick $a \in \Z^2$ such that $x_{i+1} \in \W^s_a(x_i)$ for all $0 \le i \le k-1$. By definition of the $su$-functional, we have
\[
 p_\beta(\alpha(a) x_i, \alpha(a) x_{i+1}) =\beta(a,x_i) p_\beta(x_i, x_{i+1})\beta(a,x_{i+1})^{-1}  \quad \forall i.
\]
Thus,
\[
P_\beta(\gamma) = P_\beta(\alpha(a)\gamma) = P_\beta(\alpha(na)\gamma) \quad \forall n \in \N.
\]
Since $\beta$ is continuous and $d(\alpha(na)x_i, \alpha(na)x_{i+1}) \to 0$ as $n \to \infty$, we have $P_\beta(\alpha(na)\gamma) \to id$. Therefore, $\gamma \in \ker P_\beta$, and since $S_{x_0}(\alpha)$ is generated by all stable cycles, we obtain
\[
S_{x_0}(\alpha) \subset \ker P_\beta.
\]
\end{proof}

Hence, $P_\beta: C_{x_0}(\alpha) \to D$ factors through the quotient $C_{x_0}(\alpha)/S_{x_0}(\alpha)$.

\subsection{Lyapunov cycles of $\alpha_0$ and $\hat{\alpha}$}

To show that the image of $C_{x_0}(\hat{\alpha})/S_{x_0}(\hat{\alpha})$ is small, we compare it with the corresponding set for $\alpha_0$.

It is proved in Section 7 of \cite{VinWang} that the $su$-functional of the contractible cycles of $\alpha_0$ is trivial:

\begin{prop}[Theorem 7.2, \cite{VinWang}]\label{minper}
Let $\alpha_0: \Z^2 \to D \subset \mathrm{Diff}(X)$ be a generic restriction of the Weyl chamber flow on $X = G/\Gamma$. Then
\[
C^G_{x_0}(\alpha_0)/S_{x_0}(\alpha_0)
\]
is \emph{minimally almost periodic}, i.e., it admits no continuous homomorphism into locally compact groups.
\end{prop}

Applied to the $su$-cycle functional, we have the following corollary for quotients of simple Lie group.
\begin{cor}\label{trifunc}
 Let $\alpha_0:\Z^2\to D\subset \mathrm{Diff}^\infty(X)$ be a genuinely higher rank restriction of the Weyl chamber flow on $X=G/\Gamma$, then the $su$-cycle functional $P_\beta$ of any cocycle $\beta$ of $\alpha_0$ is trivial on $C^G_{x_0}(\alpha_0)$.
\end{cor}

We can also show the same for $\alpha$-cycles.

With some abuse of notation, we lift $\hat{\alpha}$ and $\alpha_0$ to the universal cover of $X$ without changing notations. We define a natural map $P$ from topological Lyapunov paths of $\hat{\alpha}$ to Lyapunov paths of $\alpha_0$, taking an $\hat{\alpha}$ Lyapunov path with endpoints $[x_0,x_1,...,x_k]$ to the unique $\alpha_0$-Lyapunov path $[y_0,y_1,..., y_k]$ such that $y_0=x_0$, $y_k\in \tilde{W}^c_{f_0}(x_k)$ and $y_i\in \tilde{W}_{\alpha_0}^j(y_{i-1})$ if $x_i\in \tilde{W}_{\hat{\alpha}}^j(x_{i-1})$ for any $1\le i \le k$. We may similarly define an inverse map $Q$ from Lyapunov paths of $\alpha_0$ to Lyapunov paths of $\hat{\alpha}$.

Note that a priori, a Lyapunov cycle of $\hat{\alpha}$ may not be mapped to a cycle of $\alpha_0$. However, we notice that any stable cycle of $\hat{\alpha}$ must be mapped to a stable cycle of $\alpha_0$. Thus using Proposition \ref{minper}, we conclude that the contractible $\alpha_0$-cycles are also mapped to contractible $\alpha$-cycles. Proving the statement for the inverse map is harder. The proof can be found in Section 12 of \cite{VinWang}.
\begin{prop}[Theorem 12.2 \cite{VinWang}] \label{bijalpha}
 The canonical maps $P$ and $Q$ are bijections from $C^G_{x_0}(\alpha_0)$ to $C^G_{x_0}(\hat{\alpha})$ and the bijection sends $S_{x_0}(\alpha_0)$ to $S_{x_0}(\hat{\alpha})$.
\end{prop}
\begin{rmk}
 We note that in \cite{VinWang} the proof only used the smallness of the $C^0$ distance between the $\hat{\alpha}$-foliations and the $\alpha_0$-foliations, which in our setting is guaranteed by the smallness of $C^1$-distance between $f$ and $f_0$ and $C^0$-smallness of the leaf conjugacy.
\end{rmk}

As a corollary, we see that the holonomy group of $\hat{\alpha}$ is the same as the holonomy group of $\alpha_0$, thus it acts freely and transitively on the center leaves. 
\begin{prop}[Theorem 6.1 \cite{DK}; Corollary 5.6 \cite{Vinhage}]\label{hologroup}
 The group of $su$-holonomies of $\hat{\alpha}$ acts freely and transitively on a generic center leaf $Dx_0$.
\end{prop}

Therefore, the $su$-cycle functional $P_\beta$ factors through $C_{x_0}/ C^G_{x_0}=\Gamma$. However, since $G$ is semisimple with factor of rank at least $2$, it has property $(T)$, and any lattice $\Gamma$ must map virtually trivially into $D$.

\begin{prop}\label{latticefin}
 Let $\Gamma$ be a cocompact lattice in a semisimple, $\R$-split Lie group $G$ with simple components of rank at least two, then any homomorphism $\Gamma\to D$ must be virtually trivial.
\end{prop}

Therefore, Theorem \ref{cocrig} follows as a direct consequence of the propositions above. 

\begin{proof}[Proof of Theorem \ref{cocrig}]
 By Corollary \ref{trifunc}, Proposition \ref{bijalpha} and Proposition \ref{latticefin}, the image of the $su$-cycle functional must be virtually trivial. Therefore by Proposition \ref{consttrivialfunctional} the $su$-functional of the cocycle $\beta$ must be trivial up to a constant power, but since $\beta$ is $C^0$ close to a constant, the image of the $su$-functional is also $C^0$ close to identity, so it must be trivial, and thus $\beta$ must be cohomologous to the constant.
\end{proof}

\subsection{The action $\alpha$ is partially hyperbolic}\label{alphaph}

In this section, we prove the following.
\begin{thm}\label{parthyp}
 Let $f\in \mathrm{Diff}^\infty(X)$ be a $C^1$-perturbation of a generic element of the Weyl chamber flow. Let $\alpha:A\times X\to X$ be a smooth abelian action containing $f$, that is conjugate to a restriction of the Weyl chamber flow $\alpha_1: A\times X\to X$ via a  bi-H\"older conjugacy map $h$, which sends $\W^c_{\alpha_1}$ to $\W^c_{\alpha}$ and is homotopic to the identity. Then an element of $\alpha$ is partially hyperbolic if the corresponding element of $\alpha_1$ is a generic element of the Weyl chamber flow.
\end{thm}

\begin{proof}

 Let $h$ be the bi-H\"older conjugacy. For $g=\alpha(a)$, we define the topological stable foliation of $g$ to be 
 $$\W^{s,\#}_g=h(\W^s_{\alpha_1(a)}).$$ (In fact, if $\alpha_1(a)$ is not close to any of the Weyl chamber wall, then these foliations agree with $\langle \W^i_\# : \chi_i(\alpha_1(a))<0 \rangle,$
 which is the joint integration of the topological fine stable foliations of $g$ given in Section \ref{sec:choicealphafolia}.) Similarly, we define $\W^{u,\#}_g$ to be the topological unstable foliations of $g$.

 Let $\chi=\frac{\eta_1}{2} \cdot \chi(\alpha_1(a))$ where $\eta_1<1$ is bi-H\"older exponent of $h$, and $\chi(\alpha_1(a))<0$ is the contraction rate of $\alpha_1(a)$ in its stable bundle. Then by the H\"older conjugacy, we have for any $y\in \W^{s,\#}_g(x)$, \begin{equation}\label{equ:gcontract}
 d(g^n(x),g^n(y))\le e^{\chi n} d(x,y)^{\eta_1^2} 
 \end{equation} when $n\gg1$.

 Moreover, the H\"older conjugacy gives the following dynamical characterization of the topological stable and unstable foliations of $g$: 
 
 \begin{equation}\label{equ:dyncha}
 \begin{split}
 \W^{s,\#}_g(x)=\{y\in X: \limsup_{n\to \infty } \;d(g^n(x),g^n(y))=0\}\\
 \W^{u,\#}_g(x)=\{y\in X: \limsup_{n\to -\infty }\;d(g^n(x),g^n(y))=0\}
 \end{split}
 \end{equation}

 Furthermore, by the H\"older conjugacy, and using that $\alpha_1(a)$ is an isometry along $\W^c_{f_0}$, for any $x\in X,\,y\in \W^c_f(x)$, there exists $C>1$ that depends on $x,y$, such that \begin{equation}\label{equ:cendyn}
     \frac{1}{C}\le d(g^n(x),g^n(y))\le C,\quad  \forall n\in \Z.
 \end{equation}

\textbf{Step 1:} We first show that the foliation $\W^{s,\#}_g(x)$ agrees with the Pesin stable foliation at any $x\in X$ that is Pesin-regular for an ergodic measure $\nu$ of $g$. 

Let $\nu$ be an ergodic Borel measure of $g$ and let $W^{s,\nu}_g, W^{u,\nu}_g$ be the Pesin stable and unstable foliations corresponding to $\nu$ as given in Theorem \ref{pesin}. By Equation \ref{equ:gcontract} and the dynamical characterization of Pesin foliations, we have that at every $\nu$-generic point, the topological foliation $\W^{s,\#}_g$ subfoliates $\W^{s,\nu}_g$, $\W^{u,\#}_g$ subfoliates $\W^{u,\nu}_g$. 

Moreover, the Lyapunov exponent of $g$ along $\W^c_f$ is zero. Indeed, $g$ is H\"older conjugate to $\alpha_1(a)$, which is a translation along the center, so we have $d_{\W^c_f}(g^n(x),g^n(y))$ is uniformly bounded above and below as $|n|\to \infty$ for any $y\in \W^c_f(x)$, $y\neq x$. Therefore, for any $g$-ergodic measure $\nu$, the Lyapunov exponent of $Dg|_{E^c_f}$ is zero.

Therefore, comparing the dimensions of the foliations, we see that 
$$\W^{s,\#}_g = \W^{s,\nu}_g$$
for $\nu$-generic $x\in X$.

\textbf{Step 2:} We show the Lyapunov exponents along $\W^{s,\#}_g$ are $\le \chi$ for every $g$-ergodic measure.

By Equation \ref{equ:gcontract}, for any $x, y \in \W^{s,\#}_g$, we have 
$$d(g^n(x), g^n(y)) \le e^{n(\chi+\epsilon)}$$ 
for $\epsilon \ll 1$ and $n > N(x,y)$. On the other hand, let $\beta_i$ be a Lyapunov exponent of $g$ with respect to any ergodic measure $\nu$, with corresponding Oseledets bundle $E^i_g(x)$. Then, by the graph transformation argument in \cite[Theorem 3.16]{PuShergo}, there exist disks $D_i(x) \subset \W^{s,\nu}_g$ with $T_x D_i = E^i_g(x)$ and $N > 0$ such that 
$$e^{(\beta_i-\epsilon)n} \le d(g^n(x), g^n(y)) \le e^{(\beta_i+\epsilon)n}$$ 
for $n > N$ and $y \in D_i(x)$. Therefore, we have $\chi + \epsilon \ge \beta_i - \epsilon$, which implies 
$$\beta_i \le \chi + 2\epsilon$$ 
for any Lyapunov exponent $\beta_i$ along $\W^{s,\#}_g(x)$ and any $\epsilon > 0$.

\textbf{Step 3:} We use normal form theory to construct a continuous bundle $E^s_g$ tangent to $\W^{s,\#}_g$.

We consider the finest dominated splitting of $f_0$ (which may be bigger than the splitting of the root spaces) which is $$TX=E^1_{f_0}\oplus E^2_{f_0}\oplus \cdots \oplus E^c_{f_0}\oplus \cdots \oplus E^t_{f_0}.$$

Since $f$ is a $C^1$-small perturbation of $f_0$, it admits a corresponding dominated splitting, $$TX=E^1_f\oplus E^2_f\oplus \cdots \oplus E^c_f\oplus \cdots \oplus E^t_f.$$ The dominated splitting is preserved by $g$ that commutes with $f$. The non-central bundles $E^i_f, 1\le i\le t$ are continuous, and integrate to foliations $\W^i_f$ with smooth leaves. Moreover, $Df|_{E^i_f}$ satisfy the narrow band spectrum (see Section \ref{sec:normalform}) with critical regularity $<2$, since $f$ is a $C^1$-small perturbation of $f_0$.

Since $g$ is smooth and commutes with $f$, $g$ preserves the foliation $\W^i_f$. Therefore, by Theorem \ref{normalform}, there exist normal form coordinates $H_x:\W^i_f(x)\to E^i_f(x)$, such that $Q_x:=H_{g(x)}\circ g\circ H^{-1}_x$ is a linear map.

We denote $Q^{(n)}_x=H_{g^n(x)}\circ g^n\circ H^{-1}_x$, and define the subbundle $$E^{i,s}_{f,g}(x)=\{v\in E^i_f(x): \lim_{n\to \infty} \|Q^{(n)}_x(v)\|=0\}.$$ Then by the dynamical characterization of the foliations, we have $$H_x^{-1}(E^{i,s}_{f,g}(x))=\W^{s,\#}_g(x,loc)\cap \W^i_f(x,loc).$$ 
Since $\W^{s,\#}_g$ has continuous leaves, $\W^{s,\#}_g(x)\cap \W^i_f(x,loc)$ has constant topological dimension, and $H_x$ depends $C^r$ continuously on $x$, this shows that $E^{i,s}_{f,g}(x)$ is a continuous $g$-invariant bundle. 

Let $E^s_g=\oplus_{i=1}^t E^{i,s}_{f,g}$, then $E^s_g$ is a continuous $g$-invariant bundle, tangent to $\W^{s,\#}_g(x)$ at every $x\in X$.

\textbf{Step 4:} We show that contraction of $Dg|_{E^s_g}$ is uniform, then give similar estimates to the center and unstable bundle, and conclude that $g$ is partially hyperbolic.

We apply the following lemma given by the classical estimates on subadditive sequences. (See e.g. \cite{SCHREIBER1998334}.)

\begin{lem}\label{lem:subadd}
    Let $f:X\to X$ be a continuous map of a compact metric space, and let
$F:E\to E$ be a continuous linear cocycle over $f$, where
$p:E\to X$ is a continuous vector bundle over $X$.

Assume that for \emph{every} $f$-invariant ergodic probability measure $\nu$,
the top Lyapunov exponent of $F$ is
$\le \lambda .$
Then for every $\epsilon>0$, there exists $N\in \mathbb \N$ such that
\[
\|F^n(x)\|\leq e^{n(\lambda+\epsilon)}
\qquad \text{for all } x\in X, n\ge N.
\]
\end{lem}

By Step 2, the Lyapunov exponent of $Dg|_{E^s_g}$ is uniformly bounded by $\chi<0$ for any ergodic measure $\nu$ of $g$. Therefore, Lemma \ref{lem:subadd} shows that for any $\epsilon \ll 1$, there exists $N > 1$ such that
$$\|Dg^n \mid_{E^{s}_g(x)}\|\le e^{(\chi+\epsilon) n}$$
for any $n \ge N$ and $x \in X$. 

The same argument can be applied to $g^{-1}$ to construct $E^u_g$ and prove uniform expansion along $E^u_g$.

For $E^c_f$, the same argument as in Step 2 shows that the Lyapunov exponent of $Dg|_{E^c_f}$ is zero for any $g$-ergodic measure $\nu$: otherwise, by the same reasoning as in Step 2, a non-zero Lyapunov exponent implies exponential contraction or expansion along a submanifold of a center leaf, which contradicts the dynamical characterization given in Equation \ref{equ:cendyn}. Next, Lemma \ref{lem:subadd} shows that for any $\epsilon>0$, there exists $N>1$ such that $$\|Dg^n \mid_{E^{c}_f(x)}\|\le e^{\epsilon |n|}$$
for any $n\in \Z,\,|n| \ge N$ and $x \in X$.

Moreover, the dimensions of the bundle agree with the topological dimension of the foliations, which add up to be ambient dimension, so we have a $g$-invariant splitting $$TX=E^s_g\oplus E^c_f\oplus E^u_g.$$

Combining with the uniform estimates on $Dg^N$ restricted to these bundles, this shows that $g$ is partially hyperbolic.
\end{proof}

\appendix

\section{Upgrading smoothness of the conjugacy}\label{hsmooth}

In this section, we prove that the conjugacy between $f$ and an element of the Weyl chamber flow given in Equation \ref{chomg} is smooth. 

For simplicity, we give the argument for an $\R^2$ action $\alpha$ instead of the $\Z^2$-action considered in Section \ref{cocrig}. The $\Z^2$-case follows by applying the \(\mathbb R^2\)-result to the
suspension of the \(\mathbb Z^2\)-action; see the discussion preceding
Corollary 4 in \cite{KSpat}.

\textbf{Step 1:} We first show that $h_1$ restricted to the coarse Lyapunov foliations of $\alpha$ and $\alpha_1$ is $C^\infty$, following Step 4 of Section 2.2.3 in \cite{KSpat}.

By \cite[Theorem 3]{Moore}, since $\alpha_1$ is genuinely higher rank, it is ergodic with respect to the Lebesgue measure on $X$. Let $\mu=(h_1^{-1})_*(\mathrm{vol})$, we see that $\alpha(a)$ is ergodic with respect to $\mu$ for any $a\neq 0$.

For a coarse Lyapunov foliation $\W^j_0=\cap \W^s_{\alpha_1(a)}$ which is the minimal intersection of stable foliations of elements in $\alpha_1$, there exists some $b\in \R^2$ such that $\alpha_1(tb)$ has Lyapunov exponent zero along $\W^j_0$. Actually, we know that $\alpha_1(tb)$ induces isometries between the leaves $\W^j_0(x)$ and $\W^j_0(\alpha_1(tb)(x))$. For a.e. $x\in X$, and any $y\in \W^j_0(x)$, there exists a subsequence such that $\alpha_1(t_n b)(x)\to y$ in $X$, as $n\to \infty$. We may pass to a subsequence and assume that the isometries $\alpha_1(t_n b): \W^j_0(x)\to \W^j_0(\alpha_1(tb)(x))$ converge to an isometry $\W^j_0(x)\to \W^j_0(y)$. We take $G_x$ to be the closure of the group generated by the set of isometries from $\W^j_0(x)\to \W^j_0(x)$ that are limits of $\alpha_1(t_n b)$. Then we see that $G_x$ acts transitively on $\W^j_0(x)$.

After conjugating with $h_1$, we see that $h_1^{-1}G_x h_1$ also acts on the coarse Lyapunov foliations $\W^j(h_1^{-1}(x))=\cap \W^s_{\alpha(a)}(h_1^{-1}(x))$ transitively, and we shall show that for all elements of $G_x$ the action is smooth. Let $g_n=h_1^{-1}\alpha_1(t_n b)h_1$, and suppose $\alpha_1(t_n b)$ converges to $T\in G_x$. Then we have $g_n$ are smooth isometries and converge to $h_1^{-1}Th_1$ in $C^0$-topology. On the other hand, since $f\in \alpha$ satisfies the narrow band spectrum when restricted to $\W^j$, and $g_n$ commutes with $f$, by Theorem \ref{normalform}, under the normal form coordinates, $g_n$ are polynomials of bounded degree and coefficients. Therefore, their $C^0$ limit $h_1^{-1}Th_1$ is also a polynomial. This shows that $h_1^{-1}G_x h_1$ that is generated by such $h_1^{-1}Th_1$ also consists of polynomials under the normal form coordinates, which implies they are $C^\infty$-smooth actions on $\W^j(h_1^{-1}(x))$.

For $y\in X$ that is a limit of $\mu$-regular points $x_n$ in $X$, we define $G_y: \W^j_0(y)\to \W^j_0(y)$ as the $C^0$-limit set of $T_n\in G_{x_n}$. Then since $h_1^{-1}G_{x_n}h_1$ act on $\W^j(x_n)$ by polynomials with bounded degree under the normal form coordinates, the group $h_1^{-1}G_y h_1$ also consists of polynomials, and hence $h_1^{-1}G_y h_1$ are $C^\infty$ diffeomorphisms acting transitively on $\W^j(h_1^{-1}(y))$. The action of $G_y$ on $\W^j_0(y)$ is also $C^\infty$, following the same reasoning.

This shows that for any $y\in X$, $h_1: \W^j(h_1^{-1}(y)) \to \W^j_0(y)$ is $C^\infty$, since it is a conjugacy between the smooth transitive group actions by $h_1^{-1}G_y h_1$ and $G_y$. 

\textbf{Step 2:} We show that $h_1$ is smooth in $X$. In Step 1, we showed that $h_1$, and correspondingly the transfer function $H:X\to D$, are smooth along each of the coarse Lyapunov foliations. We then apply the following theorem from \cite{KSpat}.

\begin{thm}[Theorem 2.1 \cite {KSpat}]
Let $D_1,D_2,...,D_k$ be $C^\infty$ plane fields on a manifold $M$. Suppose the Lie bracket of $\sum D_i$ generates $T_x M$ at every $x\in M$, and for any $j\in \N$, the dimension of the space spanned by the commutators of length $\le j$ is locally constant. Let $P$ be a distribution on $M$ such that for any $C^\infty$ vector field $X$ tangent to any $D_i$, $X^p(P)$ is continuous or locally $L^2$ for any $p\in \N$. Then $P$ is $C^\infty$ smooth on $M$.
\end{thm}

Applying this result to the tangent spaces of the coarse Lyapunov foliations of $\alpha_1$ and the function $H:X\to D$, we see that $H$, and hence $h_1$, is $C^\infty$ smooth on $X$.

\bibliography{cenrig}{}
\bibliographystyle{acm}

\end{document}